\documentclass[12pt,reqno]{amsart}

\usepackage{amsfonts, amssymb, amsmath,
amsthm, latexsym, euscript, epic, eepic}
 \usepackage{graphicx}

\usepackage[hyperindex]{hyperref}

 \usepackage{amsrefs}

\newtheorem{theorem}{Theorem}
\newtheorem*{theorem*}{Theorem}
\newtheorem{prop}{Proposition}
\newtheorem{lemma}{Lemma}
\newtheorem{sublemma}{Sublemma}[lemma]

\theoremstyle{definition}

\newtheorem{remark}{Remark}

\def\Q{\mathcal Q}

\def\T{\mathcal{T}}
\def\DD{\mathfrak D}

\def\o{\omega}

\def\natural{\mathbb N}

\begin{document}
\author{K. D\'\i az-Ordaz}
\address{Mathematics Department, Imperial College, 180 Queen's Gate,
London SW7 2AZ} \email{karla.diaz-ordaz@imperial.ac.uk}
\urladdr{http://www.ma.ic.ac.uk/\textasciitilde kd2}

\author{M. P. Holland}
\address{Mathematics Department, University of Exeter, Exeter EX4 4QF,
UK}
\email{M.P.Holland@exeter.ac.uk}
\urladdr{http://www.secam.ex.ac.uk/people/staff/mph204}

\author{S. Luzzatto}
\address{Mathematics Department, Imperial College, 180 Queen's Gate,
London SW7 2AZ}
\email{Stefano.Luzzatto@imperial.ac.uk}
\urladdr{http://www.ma.ic.ac.uk/\textasciitilde luzzatto}


\thanks{In writing this paper K. D\'\i az-Ordaz acknowledges the support of
CONACYT Mexico.
M. Holland acknowledges the support of the EPSRC, grant no. GR/S11862/01.
S. Luzzatto acknowledges support of EPSRC grant no. GR/T0969901.
}

\title[Statistical properties of maps with critical and singular
points]
{Statistical properties of one-dimensional maps
with critical points and singularities}

\date{23 June 2006}

\begin{abstract}
  We prove that a class of one-dimensional  maps with an arbitrary
  number of non-degenerate critical and singular points
  admits an induced Markov tower with exponential return time
asymptotics.
In particular the map has an absolutely continuous invariant probability
measure with exponential decay of correlations for H\"{o}lder observations.
\end{abstract}
\maketitle

\section{Introduction and statement of results}\label{sec_intro}

It has been recognized that interval maps can exhibit a great degree
of dynamical complexity.  Indeed, within this class are the first
rigorous examples of \emph{deterministic dynamical systems} which
exhibit \emph{stochastic}  behaviour. These maps can be
characterized in terms of the existence of a \emph{mixing}
absolutely continuous (with respect to Lebesgue) invariant
probability measure (\emph{acip}).  In this paper we explore some
geometric conditions which give rise to acip's and study their
statistical properties.

\subsection{General background}
Early examples of maps where an
explicit formula for the acip can be found include the Gauss map
\cite{Gauss} and the Ulam-von Neumann transformation
\cite{UlaNeu47}. However, finding an explicit form for the density
of the corresponding acip is not possible in most cases. Instead
attention has
focussed on the existence of an acip by giving sufficient conditions
satisfied by certain classes of transformations.

Early work focussed on smooth
uniformly expanding systems \cites{Ren57,
LasYor73, HofKel82, Ryc83} with the last couple of decades really
seeing several developments in the direction of relaxing either the
uniform expansivity assumption, by allowing critical points, or the
smoothness assumption, by allowing discontinuities with possibly
unbounded derivative.

In this paper we make a further step along these
lines by considering maps which have both critical points and
discontinuities. Before stating our results we give a brief review of
the state of the field.

\subsubsection{Maps with critical points}
In the context of unimodal maps (maps with a single critical point
$c$) the existence of an \emph{acip} can be deduced from assumptions
on the derivative growth  and recurrence of the post-critical orbit
\cites{Rue77, Jak78, Mis81, ColEck83, NowStr88,
NowStr91,BruSheStr03}. If \(|(f^{n})'(f(c))|\to\infty \)
exponentially fast (known as the Collet-Eckmann condition), it was
also shown in \cite{KelNow92}, and in \cite{You92} with an
additional condition on the rate of recurrence of the critical
point, that the map is stochastic in a very strong sense: the
\emph{acip} is mixing and exhibits exponential decay of correlations
for H\"older continuous observables. The first generalization of
these results to maps with \emph{multiple critical points} was
obtained in \cite{BruLuzStr03} where the existence of an \emph{acip}
\( \mu \) was obtained under the summability condition \( \sum
|(f^{n})'(f(c))|^{-1/(2\ell-1)}<\infty \) (for every critical point)
and the assumption that all critical points have the same order \(
\ell \). Moreover, the techniques used in that paper made it
possible to show a direct link between the mixing properties of \(
\mu \), in particular the rate of decay of correlations, and the
rate of growth of \( |(f^{n})'(f(c))| \).

\subsubsection{Maps with discontinuities}

There is also an equally significant, and indeed
much longer, list of papers concerned with \emph{uniformly
expanding} maps with discontinuities (but without critical points)
in which similar results to those proved here are obtained,
see \cite{Luz05} for references.
We remark however that most results allow
discontinuities but not unbounded derivatives. This is true for
example in one of the applications in \cite{You98} in which a
strategy similar to ours is applied. As far as we know the
exponential decay of correlations for expanding Lorenz-like maps,
i.e. maps with unbounded derivatives such as the ones we consider
here but with no critical points, follows from arguments in
\cites{Kel80, HofKel82}, see also \cite{Via}.
More recently these results have been
extended in \cite{Dia06} to obtain estimates for the decay of
correlations for observables which are not H\"older continuous.

\subsubsection{Maps with critical points and discontinuities}
The main aim of this paper is to allow for the co-existence of
\emph{multiple critical points} (with possibly different critical
orders) \emph{and multiple singularities} (points at which the map
may be discontinuous and the derivative may be discontinuous or
unbounded), see Figure 1. For such maps we establish sufficient
conditions for the existence of an ergodic \emph{acip} and for
exponential mixing. We remark that although singularities contribute
to the ``expansivity'' of the system, they also give rise to
significant technical issues, in particular related to distortion
control. A noteworthy aspect of our argument is that we introduce a
unified formalism and technique for dealing with both critical and
singular points.

\begin{figure}[h]\label{map}
    \includegraphics[width=5cm]{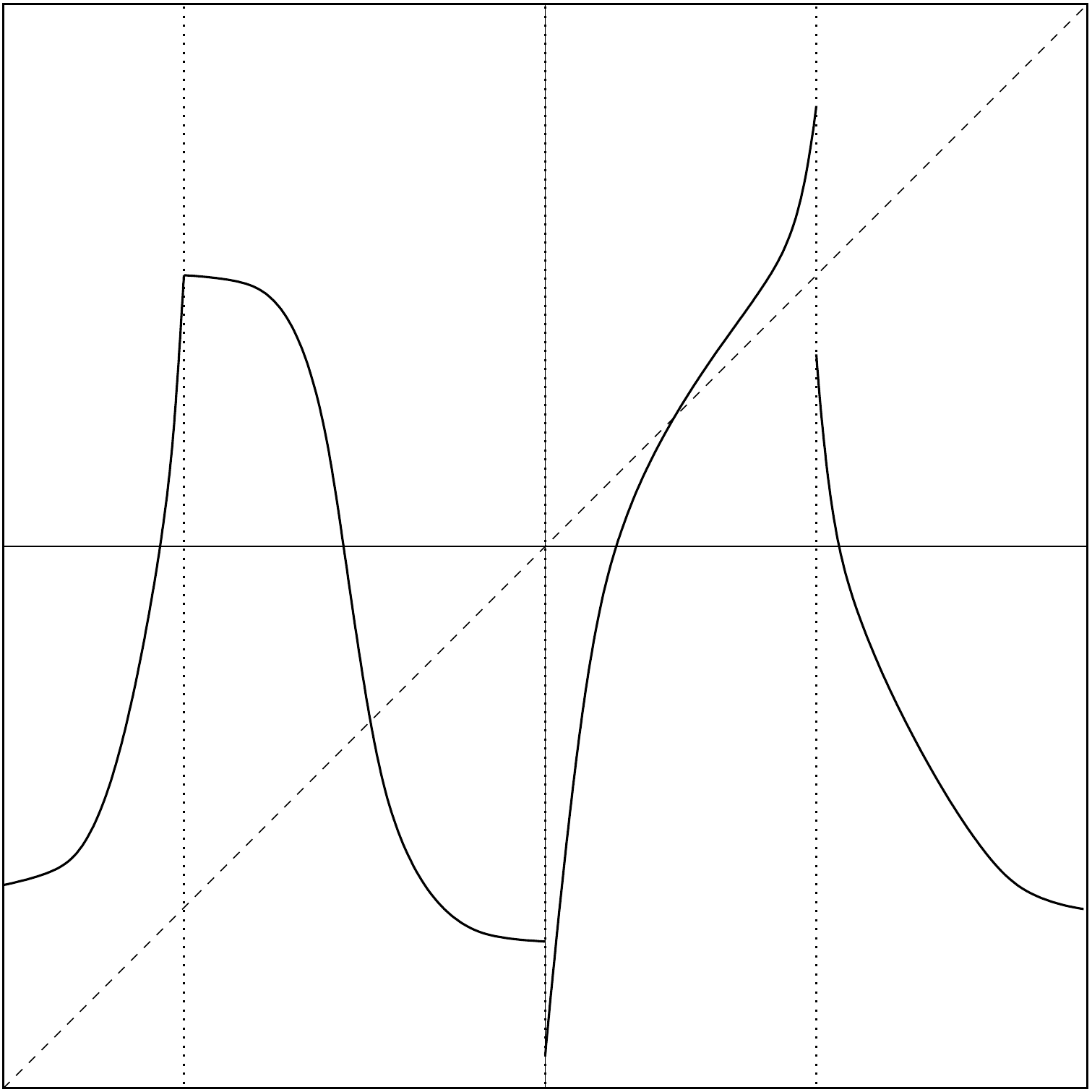}
    \caption{A map with a finite number of critical and singular points}
\end{figure}

Another class of results which should be mentioned also applies to
very general maps with non-degenerate critical and singular points
but under quite different dynamical assumptions. More specifically
the expansivity and recurrence conditions are assumed to hold
\emph{asymptotically} for \emph{Lebesgue almost every point} and not
necessarily for the critical points. It is then possible to show the
existence of an ergodic \emph{acip} \cite{AlvBonVia00} and to obtain
bounds for the rate of decay of correlations \cites{AlvLuzPindim1,
Gou04}.

\subsubsection{Induced Markov maps}
To obtain the required statistical properties
we construct a countable partition on a reference set and analyse
the recurrence time statistics through the construction of an  induced Markov
map. This strategy, which has already been applied successfully in
several contexts \cites{BruLuzStr03, AlvLuzPindim1, Hol04, AlvLuzPin,
Gou04, Dia06}, is motivated by
\cites{You98,You99} where a direct relationship between the recurrence time
statistics and mixing rates is established. A similar construction is
also carried out for general unimodal maps in \cite{DenNitUrb95}.

Inducing techniques yielding countable Markov systems have been
applied in other settings too. Jakobson has applied this idea
extensively in the context of parameter exclusion arguments for
one-dimensional maps \cites{Jak81, Jak01}. For systems admitting
infinite ergodic measures, see \cite{Aaronson} for a discussion on
\emph{Gibbs Markov maps} and their statistical properties. For maps
admitting neutral fixed points see the work of \cites{Thaler}, where
inducing schemes are used to analyse the properties of the ergodic
density. In connection with metric number theory and continued
fraction maps, inducing schemes are used in the work of
\cite{Schweiger}. A theory of \emph{Markov fibred systems} is
developed in \cites{AarDenUrb}, where inducing schemes are used to
provide results on Central Limit Theorems for a class of rational
maps admitting neutral fixed points. An approach for analysing
mixing rates using the Gibbs-Markov formalism has been developed in
\cites{Sar02, Gouezel}. Here, renewal theory techniques are applied
to establish sharp polynomial decay of correlations for Markov
systems.

A particularly interesting motivation for the construction of induced
Markov  maps is related to the development of
a \emph{thermodynamic formalism} for nonuniformly expanding maps based
precisely on the geometrical structure of (induced) Markov maps. This
is based on some recent progress on the thermodynamic formalism for
abstract countable shift spaces
\cites{Sar99, MauUrb01, BuzSar03, Sar03, JenMauUrb05}, recent
results on the relation between invariant measures for the system and
invariant measures for the Markov tower extension \cite{Zwe05}, and the
application of these results to specific classes of systems which
admit induced Markov maps such as those constructed in this paper
\cites{Yur99, PesSen05, PesZha05}.

\subsection{Non-degenerate critical and singular points.}\label{critsingdefn}
We now give the precise definition of the class of maps we consider.
Let  \( J \) be a compact interval and $f: J \to J$ a $C^2$ local
diffeomorphism outside a finite set $\mathcal{C}
\subset\mathrm{int}(J)$, of \emph{non-degenerate critical and
singular} points. These are points at which $f$ may be discontinuous
or the derivative of $f$ may vanish or be infinite. In order to
treat all possibilities in a formally unified way we consider
$\lim_{x\to c^{-}}f(x)$ and $\lim_{x\to c^{+}}f(x)$ as distinct
\emph{critical values}, thus implicitly thinking of $c^+$ and $c^-$
as distinct critical points. When referring to a neighbourhood of a
critical point, we shall always be referring to the appropriate
one-sided neighbourhood of that point. We say that the critical
points are \emph{non-degenerate} if  there exists $C>0$ and for each
$c$ there exists a constant $\ell_c\in (0,\infty) $ such that for
each $x$ in a neighbourhood of $c$ we have

\begin{equation}\label{defnCmap1}
C^{-1}|x-c|^{\ell_c}\leq |f(x)-f(c)| \leq C|x-c|^{\ell_c},
\end{equation}
and in addition we assume that
\begin{eqnarray}\label{defnCmap2}
C^{-1}|x-c|^{\ell_c-1} \!\! \!\!  &\leq |f'(x)|\leq
C|x-c|^{\ell_c-1},
\\
C^{-1}|x-c|^{\ell_c-2} \!\! \!\!  &\leq |f''(x)|\leq C|x-c|^{\ell_c-2}.
\label{defnCmap3}
\end{eqnarray}
We write $$\mathcal{C}_c=\{c:\ell_c\geq 1\}\quad\textrm{and}\quad
\mathcal{C}_s=\{c:0<\ell_c<1\},$$ to denote the set of critical and
singular points respectively. Notice that two ``distinct'' points
$c\in\mathcal{C}_c$ and $s\in\mathcal{C}_s$ may actually correspond
to the same point in $J$ for which the derivative tends to zero from
one side, and infinity from the other. We let
\[
\ell = \max_{c\in\mathcal C_{c}}\{\ell_{c}\}\quad\textrm{and}\quad
\ell^{*} = \max_{c\in\mathcal C_{s}}\{\ell_{c}\}.
\]
When there is no possibility of confusion we will often use the term
``critical point'' to refer to a point of \( \mathcal C \) without
necessarily specifying if \( c \) is really a critical point in the
traditional sense with \( \ell_{c}>1 \) or whether it is a singular
point with \( \ell_{c}\in (0,1) \) or a ``neutral'' point with \(
\ell_{c}=1 \). For
ease of exposition we assume the derivative of $f$ at the points of
discontinuity is either unbounded or zero. To accommodate bounded
derivatives, we would have to slightly modify our
argument to include the case of a return to a region where there is
a bounded discontinuity. The derivative growth and distortion
estimates would not be affected by such bounded discontinuities.

For any \( x \) let
\[
\DD(x)=\min_{c\in\mathcal{C}}|x-c|
\]
denote the distance of \( x \) from the nearest critical point, and
for small \( \delta>0 \),
let
\[
\Delta=\{x: \mathfrak D(x) \leq \delta\}
\]
denote a $\delta$-neighbourhood
of $\mathcal{C}$. For an arbitrary interval \( \omega \) we shall also
use the notation
\[
\DD(\omega) = \sup_{x\in\omega}\{\DD(x)\}
\]
to denote the distance of \( \omega \) from the critical set.

\subsection{Dynamical assumptions}\label{H1toH3}
For all initial values $x\in
J$ we let  $x_k=f^k(x)$ with  $k\in\natural$ denote the iterates of \(
x \). We formulate the following three conditions
concerning respectively the expansivity of \( f \) outside \( \Delta
\), the derivative growth and recurrence to \( \mathcal C \) of the
\emph{bona fide} critical points, and a transitivity condition on
the critical orbits.
\begin{description}
\item[(H1) Expansion outside \( \Delta \)]
\emph{There exist $\lambda>0$ and $\kappa>0$
 such that for every \( x \) and \( n\geq 1 \) such that
 $x_0=x,\dots,x_{n-1}=f^{n-1}(x)\not\in\Delta$ we have
 $$
|(f^n)'(x)|\geq \kappa \delta e^{\lambda n}.
$$
Moreover, if $x_{0}\in f(\Delta)$ or  $x_{n}\in\Delta$} we have
 $$
|(f^n)'(x)|\geq \kappa e^{\lambda n}.
$$
\item[(H2) Bounded recurrence and derivative growth along critical
orbits]\label{H2} \emph{There exists $\alpha>0$ and
$\Lambda>0$, such that for all $c\in\mathcal{C}_c$ and $\forall
k\geq 1$ we have
$$\DD(c_k)\geq\delta e^{-\alpha k}\quad\textrm{and}\quad
|(f^k)'(c_1)|\geq e^{\Lambda k}.$$}
\item[(H3) Density of preimages]
\emph{There exists  $c^*$ in $\mathcal{C}$ whose preimages are dense
in a maximal $\hat{J}\subset J$, where $\hat{J}$ is a union of
intervals (note that
the maximal property of $\hat{J}$ implies
$f^{-1}\hat{J}=\hat{J}$).
In addition, this set of preimages does not contain any
other point in $\mathcal{C}$ .}
\end{description}

We suppose throughout that
\begin{quote}\emph{
$f$ has a finite number of
non-degenerate critical and singular points and satisfies conditions
(H1)-(H3) for sufficiently small constants $\alpha>0$ and
$\delta>0$ in relation to $\Lambda,\lambda$ and $\kappa$.}
\end{quote}

It was shown in \cite{LuzTuc99} that these conditions are
satisfied for a large (positive measure) set of parameters in an
open class of one-parameter families of \emph{Lorenz-like maps with
singularities and criticalities}, in particular, our results apply
to such maps. Of course they are also satisfied by many smooth maps
but in these cases the results are already known \cite{BruLuzStr03}.

\subsection{Markov structures}
Consider the system $(f,\hat{J},m)$, where $m$, the reference measure
is taken to be Lebesgue measure.
The main result of this paper is that conditions (H1) to (H3) imply
the existence of an induced full branched Markov map with an
exponential tail of the return time function.

\begin{theorem}\label{thm_markov}
There exists a (one-sided) neighbourhood  $\Delta^{*}\subset\hat{J}$
of the critical point \( c^{*} \),  a countable
partition\footnote{Here and for the rest of the paper we will always
talk about partitions with the implicit understanding that we are
refering to partitions mod 0, i.e. up to a set of zero Lebesgue
measure.} $\mathcal{Q}$ of \( \Delta^{*} \) into
       subintervals, a function \( T: \Delta^{*} \to \mathbb N \)
       defined almost everywhere and constant on elements of
       the partition \( \mathcal Q \), and
       constants \( C, \tilde D, \gamma, >0 \) such that
       for all \( \omega\in\mathcal Q \) and \( T=T(\omega) \)
 the map \( f^{T}:\omega \to \Delta^{*} \) is a \( C^{2} \)
 diffeomorphism and satisfies the following bounded distortion
 property: for all \( x,y\in\omega \)
    \[\left| \frac{(f^{T})'(x)}{(f^{T})'(y)} -1 \right|
    \le \tilde{\mathcal{D}} |f^{T}(x)-f^{T}(y)|.
    \]
Moreover, the ``return time function'' \( T \) has an exponentially
decreasing tail:
    \[
    |\{T>n\}|<Ce^{-\gamma n}.
    \]
\end{theorem}

Our construction also gives a couple of other interesting properties
which will be used in the applications of the theorem. Namely, the
induced map is uniformly expanding in the sense that there exists
some \( \lambda' >1 \) such that for all \( \omega\in\mathcal Q \)
and all \( x,y\in\omega \)
\[
|f^{T}(x)-f^{T}(y)|\ge \lambda' |x-y|
\]
and satisfies a bounded contraction property in the sense that there
exists a constant \( K>0 \) such that for all \( \omega\in\mathcal Q
\), \( x,y\in\omega \) and \(1\leq k < T=T(\omega) \)
\[
|f^k(x)-f^k(y)|\le K |f^{T}(x)-f^{T}(y)|.
\]

\subsection{Statistical properties}
Recent results of Young \cite{You99} link the rate of decay of
the tail of the return times for an induced full branched Markov map
to several statistical properties of the original system. Combining
these general results with our main theorem and estimates we
therefore obtain the following results.

\begin{theorem}\label{thm_erg}
There exists an absolutely continuous $f$-invariant probability
measure $\mu$ which is ergodic and supported on $\hat{J}$. Moreover
$f$ is non-uniformly expanding in the sense that the Lyapunov
exponent of $\mu$ is positive, i.e. $\int\log|f'(x)|\,d\mu(x)>0.$
\end{theorem}

The fact that the measure \( \mu \) has
positive Lyapunov exponent implies a degree of \emph{sensitive
dependence on initial conditions} and \emph{stochastic-like} behaviour.
Indeed, in this case we can show that the dynamics is \emph{stochastic}
in a more precise and quantifiable way.
We recall that  a measure $\mu$ is mixing if
\begin{equation*}
\big|\mu(A\cap f^{-n}(B))-\mu(A)\mu(B)\big|\rightarrow 0
\end{equation*}
as $n\rightarrow\infty$, for all measurable sets $A,B$. This
corresponds to a property of asymptotic ``loss of memory''.  To
quantify the \emph{speed of mixing} we define,
for two arbitrary $L^{2}(\mu)$ functions  $\phi,\psi:J\to\mathbb{R}$,
 the \emph{correlation function}
\begin{equation*}
C_n(\phi,\psi,\mu)=\left |\int \psi(\phi\circ f^{n})d\mu -
\int \psi d\mu\int \phi d\mu\right |.
\end{equation*}
Notice that if \( \phi, \psi \) are characteristic functions of
measurable sets \( A, B \) this is exactly the quantity given above
in the definition of mixing. Indeed, using standard approximation
arguments,  it is well known that \( C_{n}\to 0 \) if \( \mu \) is mixing.
However, in general it is not possible to obtain a uniform bound for the
rate of decay  of \( C_n(\phi,\psi,\mu) \) if both observables belong to
a class of
functions as big as $L^2(\mu)$ or even \( L^{\infty}(\mu) \) which
includes characteristic functions, and to obtain concrete
estimates it is necessary to restrict oneself to a smaller class of
functions.

\begin{theorem}
There exists a $k\geq 1$ such that $(f^{k},\hat{J}, \mu)$ has
    exponential decay of correlations
    for  functions $\phi\in L^{\infty}$ and $\psi$
    H\"older continuous.
 \end{theorem}

The exponential rate of decay represents a particularly strong form of
mixing and indicates that, notwithstanding the presence of critical
points and the lack of smoothness due to the presence of singularities,
the system is, statistically, very similar to a uniformly expanding
system. Technically, this is a consequence of the bounded recurrence
and exponential growth conditions (H2) along the critical orbits.
Assuming weaker growth and recurrence conditions may give rise to
slower rates of mixing. This has been shown to be the case in the
smooth case  \cite{BruLuzStr03} and it would clearly be interesting to
generalize that result to allow for the presence of singularities.
Another extension of our
work could be to study statistical properties of intermittent
systems with coexisting critical points and singularities, i.e. to
include the presence of a neutral periodic point.

A further natural question concerns the convergence of averages of
functions along orbits to their expected values, in particular we can
ask if a distributional law like the
\emph{Central Limit Theorem} (CLT) holds:
for any measurable set $A\subset\mathbb{R}$
and $\phi:J\to\mathbb{R}$ with
$\int\phi\,d\mu=0,$ there exists some \( \sigma>0 \) such that
\begin{equation*}
\mu\left\{x\in J:\frac{1}{\sqrt{n}}\sum_{i=0}^{n-1}\phi\circ f^{i}(x)\in
A\right\}\to\frac{1}
{\sigma\sqrt{2\pi}}
\int_{A}e^{-\frac{t^2}{2\sigma^{2}}} dt,\quad \text{ as } n\to\infty.
\end{equation*}
We have the following result:
\begin{theorem}\label{thm_cor}
The Central Limit Theorem holds for $(f,\hat{J}, \mu)$ and any
H\"older observable \( \phi \) such that \( \phi \circ f \neq
\varphi \circ f - \varphi\) for any \( \varphi \).
\end{theorem}
Notice that Theorem \ref{thm_cor} does not depend on the mixing properties of
$f$, see \cite{MelNic04} and \cite{ChaGou06}.

\subsection{Overview  of the paper}
In Section \ref{escape_part} we introduce the interval \( \Delta^{*}
\) and begin the combinatorial construction of the induced map. We
define a neighbourhood \( \Delta^{*}\subset \Delta \) (of the
critical point \( c^{*} \), as in (H3)), with \( |\Delta^{*}| \ll
\delta \), on which we induce a Markov map. Our basic approach is to
iterate $\Delta^{*}$ under $f$ and wait for $\Delta^{*}$ to return
to $\Delta$. Using  the expansion and recurrence assumptions (H1)
and (H2) we can show that this happens after some finite number of
iterations,  $k$ say. If $f^k(\Delta^{*})$ comes close to
$\mathcal{C}_c$, then the subinterval $f^{k}(\Delta^{*})\cap\Delta$
will be contracted under iteration. For intersections with
$\mathcal{C}_s$, we do not lose expansion, but we fail to obtain
bounded distortion on future iterations of $f^{k}(\Delta^{*})$;
furthermore  $f^{k}(\Delta^{*})$ may become ``cut'' by the
singularity. To overcome these issues, we introduce a systematic
chopping procedure on $f^k(\Delta^{*})$, using a fixed partition on
$\Delta$. The resulting chopped up subintervals are iterated
independently, and at some later time these subintervals will return
back to $\Delta$. Using the fixed partition, we introduce a
combinatorial method which keeps track of their location in relation
to the critical set $\mathcal{C}$.

Following this combinatorial construction, a chopping
method is devised in such a way that $f$ (and its iterates) act
diffeomorphically, with bounded distortion on each chopped-up
subinterval. Moreover,
it could be envisaged that components may get cut too
fast in the chopping procedure, or fail to grow in size due to
frequent returns to $\mathcal{C}$. This can happen, but we
will show that \emph{on average} there is a tendency to grow in size
(at an exponential rate).

Sections \ref{sec_bind} and \ref{distest} contain two fundamental
technical results, one on the recovery of the loss of expansion (due
to the small derivative) for returns to $\mathcal{C}_c$ by shadowing
the critical orbit, and another on some global distortion bounds
which follow from the combinatorial construction. In the core
Section \ref{escest}, we estimate how long it takes on average for a
subinterval to reach \emph{large scale} and show that intervals grow
to large scale exponentially fast. In Section \ref{ret_time} we
collect all our results to prove an exponential tail estimate on the
return time function.  Finally in Section \ref{sec_tower} we explain
how the properties of the induced Markov map imply the statistical
estimates given in the other theorems.

\section{The induced map}\label{escape_part}

In this section we give the complete algorithm for the construction of
the induced Markov map as required by the statement in
Theorem~\ref{thm_markov}.  The fact that this algorithm successfully produces
an induced Markov map with the required properties is not immediate and
follows from the estimates in the following sections.

\subsection{The critical partition}
\label{critpart}
For each $c\in\mathcal{C}$ and for any integer $r\geq 1$ we let
\[
I_{r}(c)= [c+e^{-r},c+e^{-r+1}) \quad \text{and} \quad
I_{-r}(c)=(c-e^{-r+1},c-e^{-r}].
\]
We suppose without loss of
generality that \( r_{\delta} = \log \delta^{-1} \in\mathbb N\)
and,
each $c\in\mathcal{C}$ let
\begin{equation*}
\Delta_{c}=
\begin{cases}
\{c\}\cup \bigcup_{r\geq r_{\delta}+1} I_{r}(c), &\text{if $c=c^{+}$},\\
\{c\}\cup \bigcup_{r\leq -r_{\delta}-1} I_{r}(c), &\textrm{ if
$c=c^{-}$}
\end{cases}
\end{equation*}
and
\begin{equation*}
\hat \Delta_{c}=
\begin{cases}
\{c\}\cup \bigcup_{r\geq r_{\delta}} I_{r}(c), &\text{if $c=c^{+}$},\\
\{c\}\cup \bigcup_{r\leq -r_{\delta}} I_{r}(c), &\textrm{ if
$c=c^{-}$}.
\end{cases}
\end{equation*}
Notice that \( \hat\Delta_{c} \) is just \( \Delta_{c} \) union an
extra interval of the form \( I_{\pm r_{\delta}} \).
We further subdivide each $I_r\subset \Delta$ (and not the additional \(
I_{\pm r_{\delta}}\subset \hat\Delta\setminus\Delta \))
into $r^2$ intervals $I_{r,j}$,
$j\in [1,r^2]$ of equal length, this defines the \emph{critical
partition}  $\mathcal{I}$ of \( \Delta \). Finally, for each
\( r\geq r_{\delta}+1, \) and \( j\in [1,r^2] \) we let
\( \hat I_{r} \)
denote the union of \( I_{r} \) and its two neighbouring intervals. In
particular, if \( I_{r,j}=I_{r_{\delta}+1, (r_{\delta}+1)^{2}} \) is
one of the two extreme intervals of \( \Delta \), then \( \hat I_{r,j} \)
denotes the union of this interval with the adjacent intervals \(
I_{r, j-1} \) and \( I_{r_{\delta}} \) (which has not been subdivided
into smaller subintervals).

\subsection{The binding period}
\label{ss:binding}
Using the critical partition defined above, we formalize,
following \cite{BenCar85}, the notion of a
\emph{binding period} during which points in the critical region \(
\Delta \) \emph{shadow} the orbit of the critical point.
For each \( r\geq r_{\delta}+1 \), \(
I_{r}\in \mathcal I \) belonging to the component of \( \Delta \)
containing a critical point \( c\in \mathcal C \), we define
\[
p(r)=\begin{cases} 0 &\text{ if } c\in\mathcal{C}_{s}\\
\max\{k:  |f^{j+1}(x) - f^{j+1}(c)|\leq \delta e^{-2\alpha j} \
\forall\ \ x\in \hat I_{r}, \ \forall \  j  \leq k\}
&\text{ if } c\in\mathcal{C}_{c}.
\end{cases}
\]
We shall show in Section \ref{sec_bind} that the binding period is
long enough for the orbit of some point \( x \) close to a critical
point \( c\in{\mathcal C}_{c} \) to build up enough derivative growth
to more than compensate the small derivative coming from its proximity
to \( c \) in its starting position.

\subsection{Escape times}
\label{ss:escape times}
Let  \( I\subset \hat{J} \) be an arbitrary  interval with \( |I|<
\delta \).
We construct a countable partition
\( \mathcal P=\mathcal P(I) \) of \( I \) into
subintervals, which
we call the \emph{escape partition} of \( I \), and
a stopping time function \( E: I\to \mathbb N \)
constant on elements of \( \mathcal P \).
Each element \( \omega \in \mathcal P \) has some combinatorial
information attached to its orbit up to time \( E(\omega) \) and
satisfies
\[
|f^{E(\omega)}(\omega)|\geq\delta.
 \]
We define the construction inductively as follows. Fix \( n\geq 1 \)
and suppose that a certain set of subintervals of \( I \)
have been defined for which \( E < n \). Let \( \omega \) be a
component of the complement of the set \( \{x\in I: E(x) < n\} \).

\subsubsection*{Inductive assumptions}
We suppose inductively that the following combinatorial information
is also available, the meaning of which will become clear when the
general inductive step of the construction is explained below:
\begin{itemize}
    \item every iterate \( i=1,\ldots, n \) is classified as either a
    \emph{free} iterate or a \emph{bound} iterate for \( \omega \).
    \item the last free iterate before a bound iterate is called
    either an \emph{essential return} or an \emph{inessential return}.
    \item associated to each essential and inessential return there is a
    positive integer called the \emph{return depth}.
 \end{itemize}
We now consider various cases depending on
the length and position of the interval \(
\omega_{n}=f^{n}(\omega) \) and on whether \( n \) is a free or bound
iterate for \( \omega \).

\subsubsection*{Escape times}
If  \( n \) is a free time for \( \omega \) and $|\omega_n|\geq\delta$
we say that $\omega$ has \emph{escaped}.  We  let
$\omega\in\mathcal P$ and define \( E(\omega) = n
\).
We call $\omega_n$ an escape interval.

\subsubsection*{Free times}
If \( n \) is a free time for \( \omega \) and $|\omega_{n}|<\delta$
we distinguish three cases:
\begin{enumerate}
\item If $\omega_n\cap\Delta=\emptyset,$
we basically do nothing: we
do not subdivide \( \omega \) further, do not add any combinatorial
information, and define \( n+1 \) to be again a free iterate for \(
\omega \).
\item If $\omega_n\cap\Delta\neq\emptyset$ but $\omega_n$ does not intersect
more than two adjacent $I_{r,j}$'s,
we do not subdivide \( \omega \) further at this moment, but add some
combinatorial information in the sense that we
say that $n$ is an \emph{inessential}
return time with \emph{return depth} $r$ equal to the minimum \( r \)
of the intervals \( I_{r,j} \) which \( \omega_{n} \) intersects.
Moreover we  define all iterates \( j=n+1,\ldots, n+p \) as
\emph{bound iterates} for \( \omega \) (\( \omega \) does not get
subdivided during these iterates, see below), where \( p=p(r) \) is
the binding period associated to the return depth \( r \) as defined
in Section \ref{ss:binding}.
\item If $\omega_n\cap\Delta\neq\emptyset$ and
 $\omega_n$ intersects more than three adjacent
 $I_{r,j}$'s we subdivide
 \( \omega \) into subintervals $\omega_{r,j}$ in such a way that
each $\omega_{r,j}$ satisfies
$$I_{r,j}\subset f^n\omega_{r,j}\subset\hat{I}_{r,j}.$$
We say that $\omega_{r,j}$ has an \emph{essential} return at time
$n$, with return depth $r$ and define the corresponding binding period
as in the previous case.
\end{enumerate}

\subsubsection*{Bound times}
If \( n \) is a bound time for \( \omega \) we also basically do
nothing.  According to the construction above, \( n \) belongs to some
binding period \( [\nu+1, \nu+p] \) associated to a previous essential
or inessential return at time \( \nu \). So, if \( n< \nu+p \) we say
that \( n+1 \) is (still) a bound iterate, if \( n=\nu+p \) then \( n+1 \)
is a free iterate.

\subsubsection*{Returns following escape times}
\label{the return partition} The notion of an escape time is meant
to formalize the idea that the interval in question has reached
large scale, and one intuitive consequence of this is that it should
therefore ``soon'' make a full return to \( \Delta\).

\begin{lemma}\label{retpart} There exists $\delta^{*}>0$,
$t^{*}\in\mathbb{N}$ and $\xi>0$, all depending on \( \delta \),
such that for
\( \Delta^{*}=(c^{*}-\delta^{*}, c^{*}+\delta^{*}) \) a \( \delta^{*} \)
neighbourhood of the point \( c^{*} \) (recall condition H3) and for any
interval \( \omega \subset \hat J \) with \( |\omega|\geq \delta \),
there exists a subinterval \( \tilde\omega\subset\omega \) such that:
\begin{itemize}
\item $f^{t_0}$ maps $\tilde\omega$ diffeomorphically onto
$\Delta^{*}$ for some $t_{0}\leq t^{*}$,
\item $|\tilde\omega^{*}|\geq \xi |\tilde\omega|$,
\item both components of $\omega\setminus\tilde\omega$ are of size
$\geq\delta/3.$
\end{itemize}
\end{lemma}
\begin{proof}
By assumption the preimages of $c^*$ are dense in
$\hat{J}$ and do not contain any other critical point. Therefore for
any \( \varepsilon>0 \) there exists a \( t^{*} \) such
that the set of preimages
\(
\{f^{-t}(c^*): t\leq t^{*}\}
\) of the critical point \( c^{*} \) is
\emph{i)} \( \varepsilon \) dense in \( \hat J \), and
\emph{ii)}  uniformly bounded away from
$\mathcal{C}$.
Using the \( \varepsilon  \)-density and taking \( \varepsilon \)
small enough (depending on \( \delta \) but not on \( \omega \))
we can guarantee that one of these preimages belongs to \(
\omega \) and in fact we can ensure that it lies arbitrarily close to
the center of \( \omega \). Then, using that fact that these preimages
are uniformly bounded away from \( \mathcal C \) and taking \(
\delta^{*} \) sufficiently small we can actually guarantee that a
component of
\( f^{-t_{0}}(\Delta^{*}) \) for some  \(0\leq t_{0}\leq t^{*}\)
is contained the central third of \( \omega \). Since everything
depends only on a fixed and finite number of intervals and iterations
it follows that the proportion of this preimage in \( \omega \) is
uniformly bounded below.
\end{proof}

\subsubsection*{The escape partition}
We have given the complete algorithm for the construction of the
escape partition \( \mathcal P \) of the interval \( I \). The
algorithm in itself does not show that such a partition does
always exist,
indeed it may be that intervals get chopped very frequently and in
principle it may be that intervals never reach the ``large scale'' \(
\delta \) required to escape. However we shall prove that this
algorithm not only gives rise to a partition \( \mathcal P \) of \( I \)
(mod 0) but in fact escapes occur exponentially fast in the following
sense. Let
\[
\mathcal E_{n}(\omega)=\{\omega'\subseteq\omega
\text{ which have not escaped by time } n\}
\]
Then we have the following
\begin{prop}\label{basetail}
There exist constants $C_{1}> 0$ depending on \( \delta \) and \(
\delta^{*} \) and
$\gamma_1 > 0$ independent of \( \delta \) and \( \delta^{*} \)
such that for any
\( \omega\subset \hat J \) with \( \omega=\Delta^{*} \) or
\( \delta\geq |\omega|\geq \delta/3 \) we have
\begin{equation*}
    |\mathcal E_{n}(\omega)| \leq C_1 e^{-\gamma_1 n}|\omega|.
\end{equation*}
\end{prop}
Proposition \ref{basetail} will be proved in Section \ref{escest}.

\subsection{The induced Markov map}
\label{inducedMark}
We are now ready to describe the algorithm for the construction of the
final Markov induced map.
We fix \( \Delta^{*} \) as in Lemma \ref{retpart} and aim to obtain a
map  \( F: \Delta^{*}\to\Delta^{*}\) with
a partition $\mathcal{Q}$ and a return time function
$T:\mathcal{Q}\to\mathbb{N}$ constant on elements of $\mathcal{Q}$
such that \( F(\omega)=f^{T(\omega)}(\omega)=\Delta^{*} \) for every \(
\omega \in \mathcal Q \). The fact that this construction actually
yields such a partition with the required properties (distortion
bounds, tail estimates) will be verified in the following sections.

\subsubsection*{First escape partition}
First of all, starting with \( \Delta^{*} \),
we construct the escape time partition \( \mathcal P(\Delta^{*}) \) as
described in Section \ref{ss:escape times}.

\subsubsection*{Dealing with escaping components}.
Let \( \omega \in \mathcal P(\Delta^{*}) \) with some escape time \(
E(\omega) = n \).
By Lemma \ref{retpart}, we can subdivide its image
$\omega_{n}=f^{n}(\omega)$ into three pieces
\[
\omega_{n}= \omega_{n}^{L}\cup \omega_{n}^{*}\cup \omega_{n}^{R}
\]
with
\[
\omega_{n+t_{0}}^{*}=f^{n+t_{0}}(\omega)=f^{t_0}(\omega_{n}^{*})=\Delta^{*}
\]
for some $t_{0} \leq t^{*}$, and
\[
|\omega_{n}^{L}|,|\omega_{n}^{R}|>\delta/3.
\]
The interval \( \omega^{*} \) becomes, by definition,  an
element of $\mathcal{Q}$ and we define
$$
T(\omega^{*})=E(\omega)+t_0(\o)=n+t_0(\o).
$$

\subsubsection*{Iterating the argument}
The components
$\omega_{n}^{L}, \omega_{n}^{R}$ are treated as new starting
intervals and we repeat the algorithm: we construct an escape
partition of each of $\omega_{n}^{L}, \omega_{n}^{R}$  and then some
proportion of each escaping component returns to \( \Delta^{*} \)
within some uniformly bounded number of iterates.  Notice that if
either \(|\omega_{n}^{L}|\geq \delta\)  or \( |\omega_{n}^{R}|\geq
\delta \) we can skip the construction of the escape partition (or, in
some sense, this step is trivial) and immediately apply Lemma \ref{retpart}
to find a subinterval
which returns to \( \Delta^{*} \) after some finite number of iterates
bounded by \( t^{*} \). As far as the construction is concerned we
 will only apply the escape partition algorithm to
intervals \( I \) of length between \( \delta/3  \) and \( \delta \).
This explains the assumptions of Proposition \ref{basetail}.

\subsubsection*{The tail of the return times}
For \( n\geq 1 \) we let
\[
\mathcal{Q}^{(n)}=\{\omega: T(\omega)>n\}
 \]
 denote the set of intervals which arise from the construction just
 described, and which have not yet had a full return at
 time \( n\).
In Section \ref{ret_time} we will prove the following
\begin{prop}\label{main_tail}
    There exist constants $C_{2}> 0$ depending on \( \delta \) and \(
    \delta^{*} \) and
    $\gamma_2 > 0$ independent of \( \delta \) and \( \delta^{*} \)
    such that for all $n\geq 1$:
\begin{equation*}
|\mathcal{Q}^{(n)}|\leq C_{2}e^{-\gamma_{2} n}|\Delta^{*}|.
\end{equation*}
\end{prop}
This gives the required tail estimate and implies, in particular, that
Lebesgue almost every point of \( \Delta^{*} \) belongs to an interval
of the partition \( \mathcal Q \).

\section{The binding period}
\label{sec_bind}
In this short section we obtain some relatively simple but crucial
estimates related to the binding period defined in
Section~\ref{ss:binding}. In particular this shows that the binding
period defines an induced map which is
uniformly expanding on each of the
countable intervals of the critical partition \( \mathcal I \) of \(\Delta \).

\begin{lemma}\label{bindexp}
    There exists constants $\theta, \hat\theta>0$ independent of \( \delta \)
    such that for all points \( x\in \hat
    I_{r}, \) and $p=p(r) \geq 0$ we have
 \[
 |(f^{p+1})'(x)|\geq \frac{1}{\kappa} e^{\theta r}\geq
 \frac{1}{\kappa}e^{\hat\theta (p+1)}
\]
where \( \kappa>0 \) is the constant in the expansivity condition
 (H1).
\end{lemma}
\begin{proof}
We consider the singular region and the critical region separately and
then just take the minimum between the \( \theta \)'s which we obtain
in the two cases.

\subsubsection*{Estimates near \( \mathcal C_{s} \)}
This case is essentially trivial and follows immediately from the
structure of the map near the singular points. If
$x\in\hat{I}_r\subset\Delta_{c}$ with $c\in\mathcal{C}_s,$ then
$|f'(x)|\geq e^{(1-\ell_s)(r-1)},$ and  essentially any positive $
\theta <(1-\ell_s)$ will work. Since in this case \( p=0 \) the
second inequality follows as well, by taking, for example \(
\hat\theta = \theta\).

\subsubsection*{Estimates near \( \mathcal C_{c}\)}
For $x\in\hat{I}_r\subset\Delta_{c}$ with $c\in\mathcal{C}_c$ we claim
first of all that
\begin{equation}\label{bindexpeq}
 |(f^{p+1})'(x)|
\geq \frac{1}{\kappa} e^{\theta_{c}r}
\quad\text{with}\quad
\theta_{c}= 1-\frac{5\alpha\ell_{c}}{\Lambda} > 0.
\end{equation}
 Notice that \( \alpha \) is ``sufficiently small'' by assumption, we
  require here \( \alpha < \Lambda/5\ell_{c} \) for all \(
 c\in\mathcal C_{c} \). Assuming \eqref{bindexpeq}
 we can then choose
 \(
 \theta=\min\{\theta_{c}: c\in\mathcal{C}_{c}\}
 \)
 to get the first inequality in the statement of
 the lemma. We proceed to show
 \eqref{bindexpeq} and return to the second inequality in the
 statement at the end of the proof.

\subsubsection*{Bounded distortion}
First of all, by a standard argument such as that in \cite{LuzTuc99}
there is a constant
$\mathcal D_{1} $, independent of $r$ and $\delta$, such that for
all $x_1, y_1 \in f(\hat I_{r})$ and $1\leq k \leq p$,
\begin{equation}\label{binddist}
\left| \frac{(f^{k})'(x_1)}{(f^{k})'(y_1)} \right| \leq \mathcal
D_{1}.
\end{equation}

\subsubsection*{Upper bound on \( p \)}
By the definition of \( p \) we have
\( \delta e^{-2\alpha (p-1)} \geq |x_{p}-c_{p}|  \)
and, using the Mean Value Theorem and \eqref{defnCmap2}
we have
\[
\delta e^{-2\alpha (p-1)} \geq |x_{p}-c_{p}| \geq \mathcal D_1^{-1}
|(f^{p-1})'(c_{1})| \ |x_{1}-c_{1}|
 \geq  C^{-1}\mathcal D_1^{-1} e^{\Lambda (p-1)}e^{-\ell_c r}.
\]
Thus $\delta
e^{-2\alpha p}e^{2\alpha}\geq \mathcal{D}_1^{-1} C^{-1}e^{\Lambda
p}e^{-\Lambda}e^{-\ell r}$ and,
rearranging,
\begin{equation}\label{bindlen}
p\leq \frac{\log\mathcal{D}_1+\log\delta+2\alpha+\Lambda+\ell_c r +
\log C^{-1}} {\Lambda+2\alpha} \leq \frac{2\ell_{c} r}{\Lambda} \leq
\frac{2\ell r}{\Lambda},
\end{equation}
as long as we choose $\delta$ so that $r_\delta$ is sufficiently
large in comparison to the other constants, none of which depend on
\( \delta\).

\subsubsection*{Derivative estimates in terms of the return depth}
Finally, to prove \eqref{bindexpeq}, we use once again
the definition of binding, the Mean Value Theorem, and
\eqref{defnCmap1}, to get
\[
C \mathcal D_1 e^{-\ell_{c}r} |(f^p)'(x_{1})| \geq \mathcal D_1
|x_{1}-c_{1}|\ |(f^p)'(x_{1})| \geq |x_{p+1}-c_{p+1}| \geq \delta
e^{-2\alpha p}
\]
Rearranging to get a lower bound for \( |(f^p)'(x_{1})|  \), and using
 \eqref{bindlen}, we have
\begin{equation*}
|(f^{p})'(x_{1})| \geq
 C^{-1} \mathcal D_1^{-1}\delta e^{\ell_{c}r}e^{-2\alpha p}
\geq C^{-1}\mathcal D_1^{-1} \delta
e^{(\ell_{c}-\frac{4\alpha\ell_{c}}{\Lambda}) r}.
\end{equation*}
Since \( x\in \hat I_{r}\) we have \( |f'(x)|\geq
C^{-1}|x-c|^{\ell_{c}-1}\geq C^{-1}e^{-(r+2) (\ell_{c}-1)}  \) and
therefore:
\begin{align*}
|(f^{p+1})'(x)| &= |(f^{p})'(x_{1})| \cdot |f'(x)|
\\ &\geq C^{-2}e^{-2(\ell_{c}-1)} \mathcal D_1^{-1}
e^{(1-\frac{4\alpha}{\Lambda}) \ell_{c} r}e^{-r(\ell_{c}-1)}
\\ &\geq
C^{-2}e^{-2(\ell_{c}-1)} \mathcal D_1^{-1}
e^{(1-\frac{4\alpha\ell_{c}} {\Lambda}) r}
\\ &=
C^{-2}e^{-2(\ell_{c}-1)} \mathcal D_1^{-1}
e^{\frac{\alpha\ell_{c}}{\Lambda}r}
e^{(1-\frac{5\alpha\ell_{c}} {\Lambda}) r}.
\end{align*}
It therefore only remains to show that
\( C^{-2}e^{-2(\ell_{c}-1)} \mathcal D_1^{-1}
e^{\frac{\alpha\ell_{c}}{\Lambda}r} \geq 1/\kappa \) for any \( r \geq
r_{\delta} \) and any \( \ell_{c} \).
This can clearly be arranged by taking \( \delta \) sufficiently small
(and thus \( r_{\delta} \) sufficiently large) since \( \ell_{c}\geq 1 \)
and the other constants do not depend on \( \delta \).

\subsubsection*{Derivative estimates in terms of the length of the
binding period} Finally, we prove the second inequality in the
statement of the lemma. By \eqref{bindlen} we have \( r\geq \Lambda
p/2\ell \) and therefore
\[
e^{\theta r}\geq e^{\frac{\theta \Lambda }{2\ell}p}.
\]
Since the minimum binding period can be taken large by taking \(
\delta \) small this clearly implies the statement for some \(
\hat\theta \) between \( 0 \) and \( \theta\Lambda/2\ell \).
\end{proof}

\section{Distortion estimates}
\label{distest}

In this section we show that our construction yields intervals for
which some uniform distortion bounds hold. Let \( \omega\subset \hat
J \) be an arbitrary interval, \( n\geq 1 \) a positive integer such
that \( \omega \) has a sequence \( t_{0},\ldots, t_{q} \leq n \) of
\emph{free} returns to \( \Delta \) (with respective return depth
sequence \( r_{t_0},\ldots,r_{t_q} \))
followed by corresponding
binding periods \( [t_{m}+1, t_{m}+p_{m}] \),  as described above.
In particular, for \( m=1,\ldots, q \), the interval \(
\omega_{t_{m}} \) is contained in the union of three adjacent
elements of the form \( I_{r,j} \) of the critical partition \(
\mathcal I \) of \( \Delta \).

\begin{prop}\label{escapedist} There exist constants \( \mathcal
D_{\delta} \) and \( \tilde{\mathcal D}_{\delta} \) depending on \(
\delta \), and $\mathcal D$ and \( \tilde{\mathcal D} \) independent
of  \( \delta \), such that for every interval \( \omega \) and
integer \( n \) as described in the previous paragraph, for every \(
k\leq n \) and  \( x,y\in\omega \) we have
\begin{equation*}
\biggl|\frac{(f^{k})'(x)}{(f^{k})'(y)}\biggr|\leq \mathcal
D_{\delta} \quad \text{ and }\quad
\biggl|\frac{(f^{k})'(x)}{(f^{k})'(y)}-1\biggr|\leq
 \frac{\tilde{\mathcal D}_{\delta}}{|\omega_{k}|}|f^{k}(x)-f^{k}(y)|.
\end{equation*}
Moreover, the constants \( \mathcal D_{\delta} \) and \(
\tilde{\mathcal D}_{\delta} \) can be replaced by $\mathcal D$ and
\( \tilde{\mathcal D} \) under the following constraints:
\begin{enumerate}
    \item
    either we allow every
\( x,y\in \omega \) but restrict to values of  \( k\leq t_{q}+p_{q}
\);
\item
or we allow all \( k\leq n \), in particular,  \( t_{q}+p_{q}\leq k
\leq n\) but restrict to \( x,y\in\tilde\omega \)
     where \( \tilde \omega \subset \omega \) is such that
 \( \tilde\omega_{j}\cap\Delta = \emptyset \) for
 \( k>j\geq t_{m}+p_{m} \),
and  \( \tilde\omega_{k}\subset\Delta \); in this case we replace \(
\omega_{k} \) by \( \tilde\omega_{k} \) in the second inequality.

\end{enumerate}
\end{prop}

Most of the proof is devoted to proving the first inequality, then
in Section \ref{lip} we show that the second follows almost
immediately.

\subsection{Preliminary calculations}
First of all we prove the following
\begin{lemma}\label{reduction}
    There exists a constant \( \mathcal D_{2}>0 \) such that
    \[
    \log\left|
    \frac{(f^{k})'(x_0)}{(f^{k})'(y_0)}\right| \leq \mathcal D_{2}
    \sum_{j=0}^{k-1} \frac{|\omega_{j}|}{\mathfrak D (\omega_{j})}.
    \]
 \end{lemma}
\begin{proof} We carry out some relatively standard algebraic manipulations and
    then use the nondegeneracy of the critical set together with
    our assumptions about the itinerary of \( \omega \).

    \subsubsection*{Preliminary reductions}
We start by rewriting the expression for the distortion as follows.
For $j\geq 0$, let $x_j=f^{j}(x),\ y_j=f^{j}(y)$, and
$\omega_j=f^j(\omega)$. By the chain rule and the convexity of the
$\log$ function we have
\begin{equation}\label{bd1}
\log\left| \frac{(f^{k})'(x_0)}{(f^{k})'(y_0)}\right| =
\sum_{j=0}^{k-1}\log\left| 1+\frac{f'(x_j)-f'(y_j)}{f'(y_j)}
\right|\leq \sum_{j=0}^{k-1} \frac{|f'(x_j)-f'(y_j)|}{|f'(y_j)|}.
\end{equation}
Since $f$ is  $C^2$ outside $\mathcal{C}$, by the Mean Value Theorem
we can write \( |f'(x_j)-f'(y_j)|=|f''(\xi_{j})||x_j-y_j| \) for
some $\xi_{j}\in (x_j,y_j)\subset \omega_{j}$ and so
\begin{equation}\label{bd2}
\frac{|f'(x_j)-f'(y_j)|}{|f'(y_j)|}\leq
\frac{|f''(\xi_{j})|}{|f'(y_j)|} |\omega_{j}|.
\end{equation}

\subsubsection*{Using the nondegeneracy of the critical set}
Outside some fixed neighbourhood of the critical set \( \mathcal C
\) the ratio \( {|f''(\xi_{j})|}/{|f'(y_j)|} \) is uniformly bounded
above (and below). Inside such a neighbourhood we have
\eqref{defnCmap2} and \eqref{defnCmap3} and therefore, as long as
the distance of \( \xi_{j} \) and \( y_{j} \) to the critical point
\( c \) are comparable, we get
\begin{equation}\label{bd3}
\frac{|f''(\xi_{j})|}{|f'(y_j)|} \leq C^{2}
\frac{|\xi_{j}-c|^{\ell_{c}-2}}{|y_{j}-c|^{\ell_{c}-1}} \leq
\frac{\mathcal D_{2}}{\mathfrak D (\omega_{j})}
\end{equation}
for some constant \( \mathcal D_{2}>0 \). The distances of \(
\xi_{j} \) and \( y_{j} \) to the critical point are indeed
comparable, and thus
 \eqref{bd3}  holds in the situations
we are considering. To see this, we distinguish two cases. If \(
\mathfrak D(\omega_{j}) \geq \delta/2 \)  the distances clearly
are comparable since \( |\omega_{j}|\leq \delta \). On the other
hand, \( \mathfrak D(\omega_{j}) \leq \delta/2 \) implies  that \(
\omega_{j} \) is contained in at most three elements of the form \(
I_{r,j} \) of the critical partition \( \mathcal I \) and therefore
\( \mathfrak D(\omega_{j}) \gg |\omega_{j}| \).

Now,  substituting \eqref{bd3} into \eqref{bd2} and then into
\eqref{bd1} gives the statement in the lemma.
\end{proof}

\subsubsection*{Basic strategy}
By Lemma \ref{reduction}, we only need to get an upper bound,
independent of \( k \), \( \omega \) and \( \delta \), for the sum
\[
\mathcal S = \sum_{j=0}^{k-1} \frac{|\omega_{j}|}{\mathfrak D
(\omega_{j})}.
 \]
 This relies on two main ideas.
\begin{itemize}
    \item
    On the one hand we need at least \( \sum |\omega_{j}| \) uniformly
bounded. This depends on the fact that \( \omega_{j} \) is
 growing exponentially in size, the sequence \( |\omega_{j}| \) is a
 geometric increasing sequence with a uniformly bounded last term, and
 thus has a uniformly bounded overall sum.
 \item
 On the other hand, this bound
 is not sufficient since \( \mathfrak D(\omega_{j}) \) is not uniformly
 bounded below. We therefore need to use the additional information
 that \( \mathfrak D(\omega_{j}) \)  being small implies \( \omega_{j}
 \subset \Delta\) and in this case we have additional information on
 the size of \( \omega_{j} \) in relation to \( \mathfrak  D(\omega_{j}) \)
 given by our assumptions that \( \omega_{j} \) must be contained in
 some \( \hat I_{r,i} \).
 \end{itemize}
We obtain this bound in two steps.
    First of all we divide all iterates
into free and bound iterates and get estimates for the contribution
to \( \mathcal S \) of the blocks of consecutive free iterates and
bound iterates respectively.
 Secondly we add all these blocks together. This requires some
 care as we have no uniform bound on the number of such blocks,
 therefore the process of combining the estimates for each block need
 to be refined at this stage.

\subsubsection*{Free and bound iterates}
We split the sum \( \mathcal S \) into free iterates and bound
iterates to get
\begin{equation}\label{totaldistort}
\sum_{m=1}^{q} \left (\sum_{j=t_{m-1}+p_{m-1}+1}^{t_{m}-1}
\frac{|\omega_{j}|}{\mathfrak D (\omega_{j})} +
\sum_{j=t_{m}}^{t_{m}+p_{m}} \frac{|\omega_{j}|}{\mathfrak D
(\omega_{j})} \right) + \sum_{j=t_{q}+p_{q}+1}^{k-1}
\frac{|\omega_{j}|}{\mathfrak D ( \omega_{j})}.
\end{equation}
We include the return iterates, even though formally these are free,
with the bound iterates. This is not necessary but simplifies
slightly the calculations. For notational simplicity we define
$t_0+p_0+1=0$ so that in the general case the sum starts with the
free iterates \( 0,\ldots, t_{1}-1 \) preceding the first return at
time \( t_{1} \). In the special case in which the initial iterate
is already an essential return (such as when we construct the escape
partition \( \mathcal P(\Delta^{*}) \)) we have \( t_{1}=0 \) and
thus the first sum inside the parenthesis is empty for \( k=1 \).

\subsection{Distortion during free iterates}
\begin{lemma}\label{freelem}
 There exists a constant \( \mathcal D_{3}>0 \) such that for \(
 m=1,\ldots,q \),
 \begin{equation*}
   \sum_{j=t_{m-1}+p_{m-1}+1}^{t_{m}-1}
     \frac{|\omega_{j}|}{\mathfrak D (\omega_{j})}
  \leq  \mathcal D_{3} |\omega_{t_{m}}| e^{r_{t_{m}}}.
  \]
  For the last free period we have
  \[
  \sum_{j=t_{q}+p_{q}+1}^{k-1}
  \frac{|\omega_{j}|}{\mathfrak D (\omega_{j})} \leq
\frac{\mathcal D_{3}}{\delta},
 \end{equation*}
 and, restricting to a subinterval \( \tilde\omega \) as in the
 statement of Proposition \ref{escapedist} we have
 \[
 \sum_{j=t_{q}+p_{q}+1}^{k-1}
  \frac{|\tilde \omega_{j}|}{\mathfrak D (\tilde \omega_{j})} \leq
  \mathcal D_{3}.
 \]
\end{lemma}

 \begin{proof}
     For the first \( q \) free periods,
since \( \omega_{t_{m}}\subset\Delta \) the expansivity condition
(H1) implies
\begin{equation*}
|\omega_{j}| \leq \kappa^{-1} e^{ -\lambda
(t_{m+1}-j)}|\omega_{t_{m+1}}|,
\end{equation*}
and thus for $m\leq q$, using also the fact that \( \DD(\omega_{j})
\geq \delta \geq e^{-r_{t_{m+1}}}\),
\begin{equation*}
  \sum_{j=t_{m-1}+p_{m-1}+1}^{t_{m}-1}
    \frac{|\omega_{j}|}{\mathcal D (\omega_{j})}
 \leq
 \frac{|\omega_{t_{m}}|}{e^{-r_{t_{m}}}}
 \sum_{j=t_{m-1}+p_{m-1}+1}^{t_{m}-1} \kappa^{-1} e^{ -\lambda (t_{m}-j)}
\leq \mathcal D_{3}  |\omega_{t_{m}}| e^{r_{t_{m}}},
\end{equation*}
for some constant $\mathcal D_{3} >0$. This proves the first
inequality.

 For the last free period, restricting to \( \tilde\omega
\) and \( k \) as in the statement of Proposition \ref{escapedist},
we get exactly the same estimates with \( \omega_{t_{m}} \) replaced
by \( \tilde\omega_{k} \), and so, using the fact that \(
|\tilde\omega_{k}| \leq \delta \) we get the third inequality.
Without restricting to \( \tilde\omega \) and without assuming that
\( \omega_{k}\subset\Delta \), (H1) only gives a weaker expansion
estimate which implies \( |\omega_{j}| \leq \kappa^{-1} \delta^{-1}
e^{ -\lambda (k-j)}|\omega_{k}|\) and therefore a final estimate as
in the second inequality.
\end{proof}

\begin{remark}
    The dependence of \( \mathcal D_{\delta}  \) and \( \tilde{\mathcal
    D}_{\delta} \) on \( \delta \) comes entirely from the
    contribution of the last term of the  sum in \eqref{totaldistort}
    which is bounded by \( \mathcal D_{3}/\delta \) as shown in Lemma
    \ref{freelem}. For simplicity we shall now continue the proof of
    Proposition \ref{escapedist} under the assumptions which give
    bounds independent of \( \delta \). The additional statement
    follows by making minimal and obvious modifications.

    Notice also that restricting to the subinterval \(
\tilde\omega \) for the last term of   \eqref{totaldistort} does not
affect the bounds obtained for the previous terms as \(
\tilde\omega_{j}\subseteq\omega_{j} \) always implies \( {|\tilde
\omega_{j}|}/{\mathfrak  D (\tilde \omega_{j})} \leq
{|\omega_{j}|}/{\mathfrak D (\omega_{j})} \).
\end{remark}

\subsection{Distortion during binding periods}
\begin{lemma}\label{boundlem}
    There exists a constant \(
    \mathcal D_{4} >0 \) such that
\[ \sum_{j=t_{m}}^{t_{m}+p_{m}}
\frac{|\omega_{j}|}{\mathfrak D (\omega_{j})} \leq \mathcal D_{4}
|\omega_{t_{m}}| e^{r_{t_{m}}}.
\]
\end{lemma}
\begin{proof}
For each \( j \geq t_{m}\) we let \( x_{j}, y_{j}\) be two arbitrary
points in \( \omega_{j} \), and let \( c(y_{j}) \) denote the
critical point closest to \( y_{j} \), so that \(
|y_{j}-c(y_{j})|=\mathfrak  D(y_{j}) \) and,  in particular, \(
c(y_{t_{m}}) \) is the critical point involved in the binding
period. Then we can write
\begin{equation}\label{eqbind0}
 \frac{|\omega_{j}|}{\mathfrak D(\omega_{j})}\leq
 \frac{|x_j-y_j|}{|y_j-c(y_{j})|}=\frac{|x_j-y_j|}{|y_j-f^{j-t_{m}}(c(y_{t_{m}}))|}\cdot
 \frac{|y_j-f^{j-t_{m}}(c(y_{t_{m}}))|}{|y_j-c(y_{j})|}.
 \end{equation}
 We estimate the two ratios on the right hand side separately.
    \begin{sublemma}\label{eqbindsub1}
There exists a constant \( \mathcal D_{5}> 0 \) such that for each
\(  j = t_{m},\ldots,  t_{m}+p_{m} \) we have
\begin{equation}\label{eqbind1}
\frac{|x_j-y_j|}{|y_j-f^{j-t_{m}}(c(y_{t_{m}}))|} \leq \mathcal
D_{5} \frac{|\omega_{t_{m}}|}{\mathfrak{D}(\omega_{t_{m}})}.
\end{equation}
\end{sublemma}
\begin{proof}
    For \(j= t_{m} \) we have
\[
\frac{|x_j-y_j|}{|y_j-f^{j-t_{m}}(c(y_{t_{m}}))|} =
\frac{|x_{t_{m}}-y_{t_{m}}|}{|y-c(y_{t_{m}})|} \leq
\frac{|\omega_{t_{m}}|}{\mathfrak{D}(\omega_{t_{m}})}.
\]
For $t_{m}-1<j \leq t_{m}+p_{m}$ recall the bounded distortion
property during binding periods,
 \eqref{binddist}, which gives
\begin{equation}\label{eqbind0a}
\frac{|x_j-y_j|}{|y_j-f^{j-t_{m}}(c_{t_{m}})|}\leq
\mathcal{D}_1\frac{|x_{t_{m}+1}-y_{t_{m}+
1}|}{|y_{t_{m}+1}-f(c(y_{t_{m}}))|}.
\end{equation}
Since \( |\omega_{t_{m}}| \ll \mathfrak  D(\omega_{t_{m}}) \) we
have, using the non-degeneracy conditions \eqref{defnCmap2} and
\eqref{defnCmap1} on the order of the critical points,
\begin{equation}\label{eqbind0b}
|x_{t_{m}+1}-y_{t_{m}+1}| \lesssim |\omega_{t_{m}}| \mathfrak
D(\omega_{t_{m}})^{\ell_{c}-1}.
\end{equation}
and
\begin{equation}\label{eqbind0c}
|y_{t_{m}+1}-f(c(y_{t_{m}}))| \approx
|y_{t_{m}}-c(y_{t_{m}})|^{\ell_{c}} \approx \mathfrak
D(\omega_{t_{m}})^{\ell_{c}}.
\end{equation}
We use here the symbol \( \lesssim \), respectively  \( \approx \),
to mean that the left hand side is bounded above, respectively above
and below, by constants that depend only on the map \( f \).

Substituting \eqref{eqbind0b} and \eqref{eqbind0c} into
\eqref{eqbind0a} we obtain \eqref{eqbind1}.
\end{proof}

\begin{sublemma}\label{eqbindsub2}
There exists a constant \( \mathcal D_{6}>0 \) such that for each \(
j = t_{m},\ldots,  t_{m}+p_{m} \) we have
\[
   \frac{|y_j-f^{j-t_{m}}(c(y_{t_{m}}))|}{|y_j-c(y_j)|}
   \leq \mathcal D_{6} e^{-\alpha(j-t_{m})},
\]
 \end{sublemma}
\begin{proof}
 For \( j=t_{m} \) we have
 \[
 \frac{|y_j-f^{j-t_{m}}(c(y_{t_{m}}))|}{|y_j-c(y_j)|} =
 \frac{|y_{t_{m}}-c(y_{t_{m}})|}{|y_{t_{m}}-c(y_{t_{m}})|}=1.
 \]
 For $t_{m}-1<j \leq t_{m}+p_{m}$, the definition of binding period gives
 \[
 |y_j-f^{j-t_{m}}(c(y_{t_{m}}))| \leq \delta e^{-2\alpha(j-t_{m})}
 \]
 and, in conjunction with the bounded recurrence condition
(H2),
\[
|y_j-c(y_j)| \geq \delta e^{-\alpha(j-t_{m})}-\delta
e^{-2\alpha(j-t_{m})} = \delta e^{-\alpha (j-t_{m})}(1-e^{-\alpha
(j-t_{m})}) \geq \delta e^{-\alpha (j-t_{m})}(1-e^{-\alpha}).
\]
Therefore we have
\begin{equation}\label{eqbind2}
\frac{|y_j-f^{j-t_{m}}(c(y_{t_{m}}))|}{|y_j-c(y_j)|}\leq\frac{\delta
e^{-2\alpha(j-t_{m})}} {\delta e^{-\alpha (j-t_{m})}(1-e^{-\alpha})
} \leq (1-e^{-\alpha}) e^{-\alpha(j-t_{m})}.
\end{equation}
\end{proof}
Returning to the proof of Lemma \ref{boundlem}, substituting the
bounds obtained in Sublemmas \ref{eqbindsub1} and \ref{eqbindsub2}
into \eqref{eqbind0}  and letting \( \mathcal D_{7}=\mathcal
D_{5}\mathcal D_{6} \) and \( \mathcal D_{4}=\mathcal
D_{7}\sum_{i=0}^{\infty}  e^{-\alpha i}\) we get
\[
\sum_{j=t_{m}}^{t_{m}+p_{m}} \frac{|\omega_{j}|}{\mathfrak D
(\omega_{j})} \leq \sum_{j=t_{m}}^{t_{m}+p_{m}} \mathcal D_{7}
\frac{|\omega_{t_{m}}|}{\mathfrak{D}(\omega_{t_{m}})}
e^{-\alpha(j-t_{m})} \leq \mathcal D_{7}
\frac{|\omega_{t_{m}}|}{\mathfrak{D}(\omega_{t_{m}})}
\sum_{i=0}^{\infty}  e^{-\alpha i} = \mathcal D_{4}
\frac{|\omega_{t_{m}}|}{\mathfrak{D}(\omega_{t_{m}})}
\]
Finally, to get the statement in the Lemma recall that \( \mathfrak
D(\omega_{t_{m}}) \geq e^{r_{t_{m}}} \) by construction.
\end{proof}

\subsection{Combining free period and bound period estimates}
Combining the estimates of Lemmas \ref{freelem} and \ref{boundlem}
and letting \( \mathcal D_{8}=\mathcal D_{3}+\mathcal D_{4} \) gives
\[
\mathcal S = \sum_{j=0}^{k-1} \frac{|\omega_{j}|}{\mathfrak D
(\omega_{j})} \leq \mathcal D_{3}+\mathcal D_{8}\sum_{m=1}^{q}
\frac{|\omega_{t_{m}}|}{\mathfrak D (\omega_{t_{m}})} \leq \mathcal
D_{3}+\mathcal D_{8}\sum_{m=1}^{q} |\omega_{t_{m}}| e^{r_{t_{m}}}.
\]
Therefore it is sufficient to show that \( \sum_{m=1}^{q}
|\omega_{t_{m}}| e^{r_{t_{m}}} \) is uniformly bounded above. Recall
that by construction we have \( |\omega_{t_{m}}|\lesssim
e^{-r_{t_{m}}}/(r_{t_{m}})^{2} \) and so \( |\omega_{t_{m}}|
e^{r_{t_{m}}} \lesssim 1/r_{t_{m}}^{2} \) and \( \sum_{m=1}^{q}
|\omega_{t_{m}}| e^{r_{t_{m}}} \lesssim \sum_{m=1}^{q}
1/r_{t_{m}}^{2} \). This does not however imply a uniform upper
bound for the overall sum, since the sequence \( r_{t_{m}} \) is not
generally monotone and might take on the same value with unbounded
multiplicity. We therefore need to refine our estimates as follows.
First of all we subdivide the \( q \) returns under consideration
into returns with the same return depth:
\begin{equation}\label{sumsamedepth}
\sum_{m=1}^{q} |\omega_{t_{m}}| e^{r_{t_{m}}} = \sum_{r}
\sum_{t_{m}: r_{t_{m}=r} } |\omega_{t_{m}}| e^{r_{t_{m}}} = \sum_{r}
e^{r}\sum_{t_{m}: r_{t_{m}=r} } |\omega_{t_{m}}|.
\end{equation}

\begin{sublemma}
    There exists a constant \( \mathcal D_{9} \) such that
    for each \( r \) we have
    \begin{equation}\label{samedepth}
    \sum_{t_{m}: r_{t_{m}=r} }
    |\omega_{t_{m}}| \leq \mathcal D_{9}\frac{e^{-r}}{r^{2}}.
    \end{equation}
\end{sublemma}
\begin{proof}
 The statement follows from the fact that the interval \( \omega_{j} \)
 is growing exponentially fast between one return and the next.
 Indeed, by (H1) and Lemma~\ref{bindexp}, between any two consecutive
 returns we have:
\[
|\omega_{t_{m+1}}| \geq \kappa |f^{p_m + 1}(\omega_{t_m})| \geq
e^{\theta r_{t_{m}}}|\omega_{t_{m}}| \geq   e^{\theta
r_{\delta}}|\omega_{t_{m}}|.
\]
Iterating this process we get that the terms in the sum
\eqref{samedepth} form an exponentially decreasing sequence bounded
above by the length of the last term, i.e. the term with the highest
return time \( t_{m} \). For this term, our assumptions imply \(
|\omega_{t_{m}}| \leq {e^{-r}}/{r^{2}} \) thus proving the result.
\end{proof}
Substituting \eqref{samedepth} into \eqref{sumsamedepth} gives
\[
\sum_{m=1}^{q} |\omega_{t_{m}}| e^{r_{t_{m}}} = \sum_{r}
e^{r}\sum_{t_{m}: r_{t_{m}=r} } |\omega_{t_{m}}| \leq \mathcal D_{9}
\sum_{r} \frac{1}{r^{2}}
\]
which is bounded above by a uniform constant. This gives
\begin{equation}\label{sumbound}
\mathcal S = \sum_{j=0}^{k-1} \frac{|\omega_{j}|}{\mathfrak D
(\omega_{j})} \leq \mathcal D_{3}+\mathcal D_{8}\sum_{m=1}^{q}
|\omega_{t_{m}}| e^{r_{t_{m}}} \leq \mathcal D_{3}+\mathcal
D_{8}\mathcal D_{9} \sum_{r} \frac{1}{r^{2}}=:\mathcal D_{10}.
\end{equation}
This together with Lemma \ref{reduction} completes the proof of the
first inequality in Proposition \ref{escapedist}, with constant \(
\mathcal D= e^{\mathcal D_{2}\mathcal D_{10}}. \)

\subsection{Lipschitz distortion}\label{lip}
First of all, for any subinterval \( \bar\omega\subset\omega \) we
have
\[
\frac{|\bar\omega_{j}|}{\mathfrak D (\bar \omega_{j})} \leq
\frac{|\bar\omega_{j}|}{\mathfrak D (\omega_{j})} = \frac{|\bar
\omega_{j}|}{|\omega_{j}|} \frac{|\omega_{j}|}{\mathfrak D
(\omega_{j})}.
\]
For all \( x,y\in \bar\omega \), from the first inequality of
Proposition \ref{escapedist} we have
\[
\mathcal D_{2}\mathcal D_{10}\geq \log
\biggl|\frac{(f^{k})'(x)}{(f^{k})'(y)}\biggr| \geq \frac{1}{\mathcal
D_{2}\mathcal D_{10}}.
\]
Therefore, since \( x -1 \) and \( \log x \) are comparable on any
interval bounded away from 0 and \( \infty \), there exists a
constant \( \mathcal D_{11} \), depending on \( \mathcal D\),  such
that
\[
\biggl|\frac{(f^{k})'(x)}{(f^{k})'(y)}-1\biggr| \leq \mathcal D_{11}
\log \biggl|\frac{(f^{k})'(x)}{(f^{k})'(y)}\biggr|.
\]
Thus, by Lemma \ref{reduction}, and \eqref{sumbound} we have
\[
\biggl|\frac{(f^{k})'(x)}{(f^{k})'(y)}-1\biggr| \leq
 \mathcal D_{11}\mathcal D_{2} \sum_{j=0}^{k-1}
\frac{|\omega_{j}|}{\mathfrak D (\omega_{j})}\leq \mathcal
D_{11}\mathcal D_{2} \sum_{j=0}^{k-1} \frac{|\tilde
\omega_{j}|}{|\omega_{j}|}
 \frac{|\omega_{j}|}{\mathfrak D (\omega_{j})}
 \leq   \mathcal D_{11}\mathcal D_{2}  \mathcal D_{10}
 \frac{|\tilde \omega_{j}|}{|\omega_{j}|}.
\]
Finally, using once again the first inequality of Proposition
\ref{escapedist} and the Mean Value Theorem, the ratios \(
{|\bar\omega_{j}|}/{|\omega_{j}|}  \) are all uniformly comparable
to the ratio \(  {|\bar\omega_{k}|}/{|\omega_{k}|}  \). This implies
the second inequality in Proposition \ref{escapedist} and thus
completes the proof.

\section{Escape time estimates}
\label{escest}
In this section we give a complete proof of Proposition \ref{basetail}.
We assume throughout that  $\omega
=\Delta^{*}$
or \( \omega \) is an arbitrary interval with
$\delta \geq |\omega|\geq\delta/3$.
We fix from now on some \( n\geq 1 \) and recall
the definition of
\[
\mathcal E_{n}(\omega)=\{\omega'\subseteq\omega
\text{ which have not escaped by time } n\}.
\]
Each \( \omega'\in \mathcal E_{n}(\omega)  \)
has an associated sequence
\[
\nu_{1}, \nu_{2}, \ldots, \nu_{s}
\]
of \emph{essential} return times occurring before time \( n \),
and a corresponding sequence
\[
r_1,r_2,\dots,r_s
\]
of \emph{essential} return depths. (Notice that in contrast to 
Section \ref{distest} we use
a different labelling of the subscripts for the return depths).
If \( \omega=\Delta^{*} \) then \( \nu_{1}=0 \) and the sequences \(
\nu_{i} \) and \( r_{i} \) are  non-empty. In the general case
these sequences may be empty if $\omega'$ escapes without
intersecting $\Delta$.
We let
 \[
 \mathcal E_{n, R}(\omega) =\{\omega'\in\mathcal E_{n}(\omega):
r_{1}+\dots + r_{s} = R\}
 \]
 denote the union of elements of \( \mathcal E_{n}(\omega) \) with a
 given value \( R \) for the \emph{accumulated return depth}.
 Notice that \(  R \) can be 0 if the sequence
 of returns  is empty,
 but if it is non-zero then we must have  \( R\geq r_{\delta} \).
 We split the proof of Proposition \ref{basetail} into several lemmas.
First of all we show that intervals \( \omega' \) in
\(  \mathcal E_{n, R}(\omega) \) are
 exponentially small in \( R \).

\begin{lemma}\label{metestlem}
For every \( n\geq 1, R\geq 1\) and
 $\omega'\in\mathcal E_{n, R}(\omega)$
 we have
 \begin{equation}\label{len}
 |\omega'|\leq \kappa^{-1} e^{-\theta R}.
 \end{equation}
  \end{lemma}
Next, we show that  the number of intervals \( \omega' \)
    with the same accumulated return
     depth can grow  exponentially but at most with a very
     small exponential rate.
  \begin{lemma}\label{combestlem}
      There exists a constant \( \tilde\eta>0 \) which can be
      made arbitrarily small if \( \delta \) is small, such that for
      every \(n\geq 1,  R\geq 0 \)
  \[
  \sharp\{\mathcal E_{n, R}(\omega)\}\leq e^{\tilde\eta R}.
  \]
  \end{lemma}

Lemmas \ref{metestlem} and \ref{combestlem}  immediately give
 \begin{equation}\label{combined}
 |\mathcal E_{n, R}(\omega)| \leq \kappa^{-1}  e^{-(\theta-\tilde\eta) R}
 \end{equation}
 where \( \theta - \tilde\eta > 0 \) if \( \delta \) is sufficiently
 small. We can of course apply these estimates to get
  \(
  |\mathcal E_{n}(\omega)| =  \sum_{R\geq
r_{\delta}}  |\mathcal E_{n, R}(\omega)|
\lesssim \sum_{R\geq
r_{\delta}} e^{-(\theta-\tilde\eta) R}
\lesssim  e^{-(\theta-\tilde\eta) r_{\delta}}
  \)
but this bound is not good enough since it does not give an
 exponential bound in \( n \). We need to show that there is a
 relation between \( n \) and possible values of \( R \).
 This relation is given in the next two lemmas.

 \begin{lemma}\label{noreturns}
     There exists \( n_{\delta}>0 \) depending on \( \delta \), such
     that
 \[
     \omega' \in \mathcal E_{n, 0}(\omega)
\ \text{ implies} \   \omega'=\omega \text{  and }
     n\leq n_{\delta}.
 \]
   \end{lemma}

 \begin{lemma}\label{escape}
 For all \( n\geq 1 \) and \( R\geq 1 \) such that
 \( \mathcal E_{n, R} \neq
 \emptyset \) we have
 \[
 R  \geq  (n-n_{\delta})/\tilde\theta.
 \]
 \end{lemma}
Lemmas \ref{noreturns} and \ref{escape} and equation
\eqref{combined} give, for \( n> n_{\delta} \) and some constant \(
\tilde C_{1}>0 \),
\[
 |\mathcal E_{n}(\omega)|  =
 \sum_{R\geq (n-n_{\delta})/\tilde\theta}  |\mathcal E_{n, R}(\omega)|
\leq \sum_{R\geq (n-n_{\delta})/\tilde\theta} e^{-(\theta-\tilde\eta) R}
\leq \tilde C_{1} e^{\frac{(\theta-\tilde\eta) n_{\delta}}{\tilde\theta}}
e^{-\frac{\theta-\tilde\eta}{\tilde\theta} n}.
\]
Finally we  multiply the right hand side by \(
|\Delta^{*}|/\delta^{*} \) or \( |\omega|/\delta \) to get the
statement in Proposition \ref{basetail} with \(
 \gamma_{1}= (\theta-\tilde\eta)/\tilde\theta \). In the next four
 subsections we prove the four Lemmas above.

 \subsection{Escapes with essential returns: metric estimates}
 \begin{proof}[Proof of Lemma \ref{metestlem}]
 Let \( \omega'\in\mathcal E_{n,R} \)
 with an associated sequence \( \nu_{1},\ldots,\nu_{s} \) of essential
 returns corresponding return depths
 \( r_{1},\ldots, r_{s} \). Then, by construction we have
 \[
 |\omega'_{\nu_{s}}| \leq e^{-r_{s}}.
 \]
 Moreover, using the expansion after the binding
 periods from Lemma \ref{bindexp} and the expansion during the free
 periods (H1), including the fact that each free period under
 consideration ends in a return to \( \Delta \), we have
 \[
 |(f^{\nu_{1}})'(x)|\geq \kappa e^{\lambda \nu_{1}} \geq \kappa
 \]
 for all \( x\in \omega' \) and
 \[
 |(f^{\nu_{i+1}-\nu_{i}})'(x_{\nu_{i}})|\geq e^{\theta r_{i}}
 \]
 for all \( i=1,\ldots, s-1 \) and \( x_{\nu_{i}}\in \omega'_{\nu_{i}}
 \).
 This gives
 \[
 |(f^{\nu_{s}})'(x)|\geq \kappa e^{\theta (r_{1}+\dots + r_{s-1})}
 \]
for all \( x\in\omega' \). Thus, by the mean value theorem we have
\[
|\omega'| \leq \kappa^{-1} e^{-\theta (r_{1}+\dots + r_{s-1})}
|\omega'_{\nu_{s}}| \leq \kappa^{-1} e^{-\theta (r_{1}+\dots + r_{s})}.
\]
 \end{proof}

 \subsection{Escapes with essential returns: combinatorial estimates}
 \label{combest}

 \begin{proof}[Proof of Lemma \ref{combestlem}]
 We divide the proof into two steps.
The first one is purely combinatorial and bounds the number of
theoretically possible combinatorially distinct elements. The second
one relies on the construction and bounds the number of possible
elements with the same combinatorics.

 \subsubsection*{Cardinality of possible sequences}

 Let $N_{R,s}$ denote the number of integer sequences
 $(t_1,\ldots,t_s)$, $t_i\geq r_{\delta}$ for all $i$, $1\leq i \leq s$, such
 that $\sum_{i=1}^{s}t_i = R$. The number \( N_{R,s} \) of such
 sequences is the same
 as the number of ways to choose $s$ balls from a row of $k+s$
 balls, thus partitioning the remaining $k$ balls into at most $s+1$
 disjoint subsets.  Therefore we have
 \[
 N_{R,s}\leq\begin{pmatrix}
 R+s\\ s
 \end{pmatrix}
 =\begin{pmatrix} R+s\\ R
 \end{pmatrix}.
 \]
Since each term \( t_{i}\geq r_{\delta} \) we must always have \( s\leq
  R/r_{\delta} \), and, since the right hand side above  is monotonically
 increasing in $s$,  writing \( \eta=1/r_{\delta} \) for simplicity, we
 get
 $$
 N_{R,s}\leq
 \begin{pmatrix}
 (1+\eta)R \\ R
 \end{pmatrix}
 =\frac{[(1+\eta)R]!}{(\eta R)! R!}.
 $$
 Using Stirling's  formula
 $k!\in [1,1+\frac{1}{4k}]\sqrt{2\pi k}k^{k}e^{-k}$
 we obtain
 \begin{align*}
 N_{R,s}
 &\leq\frac{[(1+\eta)R]^{(1+\eta)R}}{(\eta R)^{\eta R}R^R}
 =(1+\eta)^{(1+\eta)R}\eta^{-\eta R}\\
 & \leq \exp\{(1+\eta)R\log(1+\eta) -\eta R\log\eta\}
 \leq \exp\{((1+\eta)\eta -\eta \log\eta)R\}=e^{\eta_{1} R}
 \end{align*}
 where $\eta_{1} =
 \bigl((1+\eta)\eta -\eta \log\eta\bigr)$ can be made arbitrarily small
 with \( \delta \).
Now, letting \( N_{R} \) denote the number of all possible integer sequences
$(t_1,\ldots,t_s)$, $t_i\geq r_{\delta}$ for all $i$, $1\leq i \leq s$, such
that $\sum_{i=1}^{s}t_i = R$, for all possible values of \( s \) we
have
  \begin{equation}\label{eqsum}
  N_{R}= \sum_{s=1}^{\eta R} N_{R,s}\leq
  \sum_{s=1}^{\eta R} e^{\eta_{1} R} =
  \eta R
  e^{\eta_{1} R}.
  \end{equation}

 \subsubsection*{Multiplicity of intervals sharing the same sequence}
 From the
 subdivision procedure described in Section \ref{escape_part},
 the only way in which several elements can share the same
 combinatorics is by having the same sequence of return depth
 $r_1,r_2,\ldots, r_s$ but being associated to different critical
 points and to different intervals \(
 I_{r_{i}, j} \) (recall that each element \( I_{r} \) in the critical
 partition is subdivided into \( r^{2} \) subintervals of equal
 length).
 Therefore, letting \( N_{c} \) denote the number of (one-sided)
 critical points and \( \mathcal E_{n, r_{1},\ldots, r_{s}} \)
 the set of  escaping intervals having the same given
sequence \( r_{1},.., r_{s} \)
 of return depths, we have
 \[
\sharp \mathcal E_{n, r_{1},\ldots, r_{s}}\leq
\prod_{j=1}^{s} N_{c} r^{2}_{j}.
 \]
To get an upper bound for the right hand side we take logs
and obtain
 \begin{equation}\label{eqprod}
\log\prod_{j=1}^{s} N_{c} r^{2}_{j}
 =\sum_{j=1}^{s} \log N_{c} r^{2}_{j}
=s\log N_{c} + \sum_{j=1}^{s} 2\log r_{j}.
 \end{equation}
 Since \( R\geq s r_{\delta} \) we have
 \begin{equation}\label{eqprod1}
 s\log N_{c} \leq \frac{\log N_{c}}{r_{\delta}} R = \eta_{2}R
 \end{equation}
 where \( \eta_{2} \) can be made arbitrarily small by taking \(
 \delta \) small. Moreover since \( r_{j}\geq r_{\delta} \) we also have
 \begin{equation}\label{eqprod1a}
 \frac{2 \log r_{j}}{r_{j}} \leq \frac{2 \log r_{\delta}}{r_{\delta}} =
 \eta_{3}
 \end{equation}
 and therefore
 \begin{equation}\label{eqprod2}
 \sum_{j=1}^{s} 2\log r_{j} \leq \eta_{3}\sum_{j=1}^{s} 2 r_{j} =
 \eta_{3} R.
 \end{equation}
 where \( \eta_{3} \)
 can be made arbitrarily small by taking  $\delta$ small.
 Substituting \eqref{eqprod1} and \eqref{eqprod2} into \eqref{eqprod} gives
 \[
 \log\prod_{j=1}^{s} N_{c} r^{2}_{j} \leq (\eta_{2}+\eta_{3}) R
 \]
 or
 \begin{equation}\label{eqprod3}
     \sharp\{\mathcal E_{n, r_{1},\ldots, r_{s}}\} \leq
\prod_{j=1}^{s} N_{c} r^{2}_{j} \leq e^{(\eta_{2}+\eta_{3}) R}
 \end{equation}

\subsubsection*{Final estimate}
The final upper bound for the number of possible intervals in \(
\mathcal E_{n, R}(\omega) \) is therefore just given by multiplying
the number of possible sequences with a given total return depth \( R \)
by the number of possible partition elements which can potentially
have exactly such a value as their total return depth. Thus,
multiplying  \eqref{eqsum} by \eqref{eqprod3} we get
 $$
 \sharp\{\mathcal E_{n, R}(\omega) \}\leq
\eta Re^{(\eta_{1} +\eta_{2}+ \eta_{3} )R}
 \leq e^{(\eta_{1}+\eta_{2}+\eta_{3})R+
 \log\eta + \log R }
 \leq e^{(\eta_{1}+\eta_{2}+2\eta_{3})R}.
 $$
 In the last inequality we have used the fact that \( \log \eta < 0 \)
 since \( \eta = 1/r_{\delta} \) is small, and that \( \log R <
 \eta_{3} R \) from \eqref{eqprod1a} and the fact that \( R\geq
 r_{\delta} \). Thus we get the statement in Lemma \ref{combestlem}
 with \( \tilde\eta = \eta_{1}+\eta_{2}+2\eta_{3} \) where \(
 \tilde\eta \) can be made arbitrarily small if \( \delta \) is chosen
 small enough.

 \end{proof}

 \subsection{Escapes with no essential returns}
 \begin{proof}[Proof of Lemma \ref{noreturns}]
 First we show that     \( \omega'=\omega \) and then estimate \(
 n_{\delta} \). For convenience we introduce here a couple of
 constants which will be used here and in the next subsection. Let
 \[
 \hat\lambda= \min\{\lambda, \hat\theta\}
 \quad \text{ and }
 \quad \tilde{\lambda} = \min\{\lambda, \hat\theta, \Lambda\}.
 \]

\subsubsection*{Claim: \( \omega'=\omega \)}
This  follows directly from the
construction. Indeed, suppose by contradiction
that \( \omega'\subset\omega \). This
would mean that \( \omega \) had been
chopped at some time \( \nu \leq  n \) for which \(
\omega_{\nu}\cap\Delta\neq\emptyset \). However, at this time all
subintervals of \( \omega \) arising from this chopping procedure
would qualify as having had an essential return. This is true even if \(
\omega_{\nu} \) is not strictly contained in \( \Delta \) since either
the components of \( \omega_{\nu}\setminus\Delta \) are smaller than
\( \delta \) and are therefore, by construction, attached to their
adjacent elements which fall inside \( \Delta \), or they are bigger
than \( \Delta \) which implies that \( |\omega_{\nu}|\geq \delta \)
implying that \( \nu \) is an escape time for \( \omega \). This
contradicts our assumption that \( \omega' \) does not escape before
time \( n \). Thus \( \omega'=\omega \).

\subsubsection*{No inessential returns}
To obtain a bound for \( n \) we suppose first of all that there are
no (inessential) returns, i.e. \( \omega \) stays outside \( \Delta \)
up to time \( n \). Then, by conditions (H1) we have
\( |(f^{n})'(x)|\geq \kappa\delta e^{\lambda n}  \) and so, by the
mean value theorem and using the fact that \( |\omega|\geq \delta/3 \)
and that \( |\omega_{n}|\leq \delta \) we get
\[
\delta \geq |\omega_{n}|\geq \kappa\delta e^{\lambda n} |\omega|
\geq \kappa\delta^{2} e^{\lambda n}/3.
\]
This gives \( e^{\lambda n}\leq \frac{3}{\kappa\delta} \) and, solving
for \( n \) we obtain
\( n\leq \frac{1}{\lambda}\log \frac{3}{\kappa\delta}. \)

\subsubsection*{If \( n \) is a free iterate}
If there are any returns to \( \Delta \) we distinguish two further
cases. If \( n \) is a free iterate, then we can combine the binding
period expansion estimates and condition (H1) to get
\(  |(f^{n})'(x)|\geq \kappa\delta e^{\hat\lambda n}  \).
Then reasoning
exactly as above we get   \( n\leq \frac{1}{\hat \lambda}
\log \frac{3}{\kappa\delta} \).

\subsubsection*{If \( n \) is a bound iterate}
If \( n \) is not a free iterate we have the additional minor
complication of not being able to use the binding period estimate
for the last incomplete binding period. Let \( \nu<n \) be the last
inessential return before \( n \) such that \( n \) belongs to the
binding period which follows the return at time \( \nu \). Then, by
the calculation above we have
\(  |(f^{\nu})'(x)|\geq \kappa\delta e^{\hat\lambda \nu}  \) and, in
particular,
\[
|\omega_{\nu}|\geq  \kappa\delta^{2} e^{\hat\lambda \nu}/3
\geq  \kappa\delta^{2}/3.
\]
Therefore, for the return at time \( \nu \) to be inessential it
cannot be too deep, i.e. it must have a return depth \( r \leq \log
\kappa\delta^{2}/3 \), and therefore the derivative at points \(
x_{\nu}\in\omega_{\nu} \) must be of the order of \( e^{r}\approx
\kappa\delta^{2}/3  \). Therefore \(  |(f^{\nu+1})'(x)| \gtrsim
\delta^{3} e^{\hat\lambda \nu} e^{r} \kappa/3   \). During the
remaining iterates, the bounded distortion during the binding
periods implies that the derivative is growing exponentially fast at
rate \( \Lambda \), from condition (H2). We therefore have \(
|(f^{n})'(x)|\gtrsim\delta ^{3} e^{\tilde{\lambda} n}\). Arguing
once again as in the previous case we then have
\[
\delta \geq |\omega_{n}|\gtrsim \delta^{3} e^{\tilde{\lambda} n}
|\omega| \gtrsim \delta^{4} e^{\lambda n}.
\]
Solving for \( n \) once again we get the result.
 \end{proof}

\subsection{Escape times and return depths}

\begin{proof}[Proof of Lemma \ref{escape}]
Let \( \omega'\in\mathcal E_{n, R}\) and let
\( \nu_{1},\ldots, \nu_{s} \)
    be the sequence of essential returns, \(
    r_{1},\ldots, r_{s} \) the corresponding sequence of return
    depths, and \( p_{1}, \ldots, p_{s} \) the
corresponding binding periods. To simplify the notation we let
\( \nu_{s+1}:=n \).

Lemma \ref{noreturns}  implies
\( \nu_{1}\leq n_{\delta} \) and therefore it is sufficient to prove
that there exists  a constant \(
\tilde\theta \) such that for all  \( i=1,\ldots, s \) we have
\begin{equation}\label{consret}
\nu_{i+1}-\nu_{i} \leq \tilde\theta r_{i}.
\end{equation}
Indeed, this immediately gives \( n\leq n_{\delta}+ \tilde\theta R \)
which is equivalent to the statement in the Lemma.

\begin{sublemma} For all  \( i=1,\ldots, s \) we have
    \begin{equation}\label{derest}
    \inf_{x\in\omega'} |(f^{\nu_{i+1}-\nu_{i}})'(x)|  \leq \delta
    e^{r_{i}} r_{i}^{2}.
    \end{equation}
\end{sublemma}
\begin{proof}
    Since \( \nu_{i} \) is an essential return there
    exists an interval \( \hat\omega' \) with
    \( \omega'\subseteq\hat\omega'\subseteq \omega \) such that
    the image \( \hat\omega'_{\nu_{i}} \) at time \( \nu_{i} \)
    contains a unique interval $I_{r,j}=I_{r_i,j_i}$ and therefore
    \( |\hat\omega'_{\nu_{i}} |\geq e^{-r_{i}}/r_{i}^{2} \).
    Since we are assuming here that \( \hat\omega' \) does not escape
    before or at time \( \nu_{i+1} \) we also have
    \( \delta \geq |\hat\omega'_{\nu_{i+1}}|  \). Therefore, by the
     mean value theorem we have
    \[
    \delta \geq |\hat\omega'_{\nu_{i+1}}| \geq
    \inf_{x\in\hat\omega'} |(f^{\nu_{i+1}-\nu_{i}})'(x)|
    \ | \hat\omega'_{\nu_{i}}|
    \geq \inf_{x\in\hat\omega'} |(f^{\nu_{i+1}-\nu_{i}})'(x)|
    e^{-r_{i}}/r_{i}^{2}
    \]
    which gives \eqref{derest}.
 \end{proof}

\begin{sublemma} For all  \( i=1,\ldots, s \) we have
    \begin{equation}\label{growth}
\inf_{x\in\omega'} |(f^{\nu_{i+1}-\nu_{i}})'(x)| \geq
\delta e^{\hat\lambda (\nu_{i+1}-\nu_{i})}.
\end{equation}

\end{sublemma}
\begin{proof}
    Let \( \rho_{1},
    \rho_{2},\ldots, \rho_{t(i)} \) denote the sequence of return depths
    corresponding to inessential returns occurring between
    \( \nu_{i} \) and \( \nu_{i+1} \) and \( p_{1}, \ldots, p_{t(i)} \)
    the corresponding binding periods. Then, combining Lemma \ref{bindexp}
    and the expansion (H1) during free iterates, keeping in mind that
    every free iterate under consideration ends in \( \delta \), we have
\( |(f^{\nu_{i+1}-\nu_{i}})'(x)| \geq
e^{\hat\lambda (\nu_{i+1}-\nu_{i})}  \) for all \( i= 1,\ldots, s-1 \)
and \( |(f^{\nu_{i+1}-\nu_{i}})'(x)| \geq
\delta e^{\hat\lambda (\nu_{i+1}-\nu_{i})}  \) for \( i=s \).
     The factor \( \delta \) in the second
    inequality is due to the fact that \( n \) is not necessarily a return.
 \end{proof}

Returning to the proof of the Lemma, we combine
 \eqref{growth} into \eqref{derest} to get
 \[
 e^{r_{i}} r_{i}^{2}\geq  e^{\hat\lambda (\nu_{i+1}-\nu_{i})}
 \]
 and solving for \(
\nu_{i+1}-\nu_{i} \) gives
\[
\nu_{i+1}-\nu_{i} \leq \frac{1}{\hat\lambda} (r_{i}+2\log r_{i})
\leq \frac{1+\eta_{3}}{\hat\lambda} r_{i}
\]
where \( \eta_{3} \) is the constant in \eqref{eqprod1a}.
Thus the Lemma follows with \( \tilde\theta = (1+\eta_{3})/\hat\lambda \).
\end{proof}

\section{Return Time Estimates}\label{ret_time}
We are now ready to prove Proposition \ref{main_tail}. We remark that
from this point onwards we shall \emph{consider \( \delta \) to be
fixed once and for all}. In particular we shall make use of the
distortion constant \( \mathcal D_{\delta} \) from Proposition
\ref{escapedist}; this constant depends on \( \delta \) but this will
not cause any problems as we shall not impose any additional
conditions on the size of \( \delta \).

We proceed initially as in the proof of the analogous estimate for the
escape time, although the argument here is more probabilistic because
we do not have such sharp control of the combinatorics.
We consider a fixed  \( n\geq 1 \) and recall the definition of
\[
\mathcal{Q}^{(n)}=\{\omega: T(\omega)>n\}
 \]
as the set of intervals which have not yet had a full return at
 time \( n\).
By construction, each \( \omega\in \mathcal Q^{(n)} \)
is contained in a nested sequence of intervals
\[
\omega\subset\omega^{(s)}\subset\omega^{(s-1)}
\subset\ldots\subset\omega^{(1)} \subset \Delta^{*}
\]
corresponding to escape times $E_1,\ldots,E_s$, such that \(
|f^{E_i}(\omega^{(i)})|\geq\delta \) for  \( i=1\ldots, s \). This
sequence is empty for those elements of \( \mathcal Q^{(n)} \) which
have not had any escape before time \( n \) (such as those which
start very close to the critical point). For \( s = 0,\ldots, n \)
(clearly there cannot be more than \( n \) escapes) we let \(
\mathcal{Q}^{(n)}_s \) denote the collection of intervals in \(
\mathcal{Q}^{(n)} \) which have exactly \( s \) escapes before time
\( n \). Then for   a constant $\zeta\in (0,1)$ whose value will be
determined below, we write
\begin{equation}\label{set_split}
|\Q^{(n)}| =\sum_{s\leq n}|\Q_s^{(n)}|= \sum_{s\leq \zeta
n}|\Q_s^{(n)}| + \sum_{\zeta n<s \leq  n}|\Q_s^{(n)}|.
\end{equation}
This corresponds to distinguishing those intervals which
have had lots of escape times and those
that have had only a few . The
Proposition then follows immediately from the following two lemmas.

 \begin{lemma}\label{fewescapes} There exists a constant \( C_{3}>0 \)
     depending on \( \delta \) and \( \delta^{*} \) and a constant \(
     \gamma_{3}>0 \) independent of \( \delta  \) and \( \delta^{*} \)
     such that, for all \( \zeta>0 \) sufficiently small and all \(
     n\geq 1 \) we have
\[     \sum_{0\leq s\leq \zeta n}|\Q_s^{(n)}|  \leq C_{3}
e^{-\gamma_{3} n} |\Delta^{*}|.
\]
  \end{lemma}

  \begin{lemma}\label{manyescapes}
There exists a constant \( C_{4}>0 \)
depending on \( \delta \)  and \( \delta^{*} \) and a constant
\(   \gamma_{4}>0 \) independent of \( \delta  \) and \( \delta^{*} \)
such that, for all \( \zeta>0 \) sufficiently small and all \(
  n\geq 1 \) we have
  \[     \sum_{\zeta n<s \leq  n}|\Q_s^{(n)}| \leq
  C_{4}  e^{-\gamma_{4} n} |\Delta^{*}|.
  \]
   \end{lemma}

\subsection{Returns after few escapes}
\begin{proof}[Proof of Lemma \ref{fewescapes}]

Starting with the escape partition on \( \Delta^{*} \), letting \(
t_{1}\geq 1 \) be an integer, we write
\[
\{E_{1}=t_{1}\}
\]
for the set of points in \( \Delta^{*} \) which belong to intervals
which have a first escape at time \( t_{1} \). By Proposition
\ref{basetail}, using the fact that \( |\Delta^{*}|=\delta^{*}
\),   we have
\[
|\{E_{1}=t_{1}\}| \leq |\mathcal E_{t_{1}}(\Delta^{*})|\leq
C_{1}e^{-\gamma_{1}t_{1}} =
\frac{C_{1}}{\delta^{*}} e^{-\gamma_{1}t_{1}} |\Delta^{*}|.
\]
We then repeat the escape partition construction on each component of
\( \{E_{1}=t_{1}\} \). We let
\[
\{E_{2}=t_{2}+t_{1}: E_{1}=t_{1}\}
\]
denote the set of points which belong to an interval which has a first
escape at time \( t_{1} \) and a second escape at time \( t_{2}+t_{1}
\), i.e. \( t_{2} \) iterates after the first escape. Proposition
\ref{basetail} gives an estimate for the size of
\( \{E_{2}=t_{2}+t_{1}: E_{1}=t_{1}\} \) \emph{at time} \( t_{1} \):
if \( \omega \) is an escape component belonging to
\( \{E_{1}=t_{1}\} \) then
\[
|\mathcal E_{t_{2}}(f^{t_{1}}(\omega))| \leq C_{1}e^{-\gamma_{1} t_{2}}
\leq \frac{C_{1}}{\delta} e^{-\gamma_{1} t_{2}} |(f^{t_{1}}(\omega))|.
\]
Using the bounded distortion property from Proposition \ref{escapedist}
this ratio is preserved for the initial interval \( \omega \) up to
the bounded distortion constant \( \mathcal D_{\delta} \) to get
\[
|\{E_{2}=t_{2}+t_{1}: E_{1}=t_{1}\}| \leq \frac{C_{1}\mathcal
D_{\delta}}{\delta} e^{-\gamma_{1}t_{2}} |\{E_{1}=t_{1} \}|
\leq \frac{C_{1}^{2}\mathcal
D_{\delta}}{\delta\delta^{*}} e^{-\gamma_{1}(t_{1}+t_{2})}|\Delta^{*}|.
\]
 Continuing in this way we get
 \[
 |\{E_{s}=t_{1}+\cdots+t_{s}: E_{1}=t_{1}, E_{2}=t_{2},\ldots,
 E_{s-1}=t_{s-1}\}| \leq \frac{C_{1}^{s}\mathcal
D_{\delta}^{s-1}}{\delta^{s-1}\delta^{*}} e^{-\gamma_{1}E_{s}}|\Delta^{*}|.
 \]
Letting \( t_{s+1}=n-E_{s} \) we then apply one more iteration of this
formula to get
\begin{equation*}
|\Q^{(n)}_s(t_1,\ldots,t_{s+1})|\leq \frac{C_1^{s+1}\mathcal
D_{\delta}^{s}}{\delta^{s}\delta^{*}} e^{-\gamma_1 n}|\Delta^{*}|.
\end{equation*}
where
\[
\Q^{(n)}_s (t_1,\ldots, t_{s+1})=\{\omega\in\Q^{(n)}_s :
E_{i+1}(\o)-E_{i}(\o) = t_{i}, i= 1,\ldots,  s, n-E_{s}=t_{s+1}
\}.
\]
Thus for each \( s \) we have
\[
|\mathcal Q^{(n)}_{s}| \leq  \sum_{\overset{(t_1,\ldots,t_{s+1})}{\sum t_j=n}}
    |\Q^{(n)}_s(t_1,\ldots, t_{s+1})|
  \leq
  N_{n,s+1} \frac{C_1^{s+1}
    \mathcal D_{\delta}^{s}}{\delta^{s}\delta^{*}} e^{-\gamma_1 n}|\Delta^{*}|
\]
where $N_{n,s+1}$ is the number of possible sequences
$(t_1,\ldots, t_{s+1})$ such that $\sum t_j=n$, and therefore
    \begin{equation*}
  \sum_{0\leq s\leq\zeta n}|\Q_s^{(n)}|
    \leq \sum_{s\leq\zeta n}N_{n,s}
    \frac{C_1^{s+1}\mathcal
    D_{\delta}^{s}}{\delta^{s}\delta^{*}} e^{-\gamma_1 n}|\Delta^{*}|
    \end{equation*}
Using exactly the same counting argument used in the proof of Lemma
\ref{combestlem} we can choose \( \hat\zeta >0  \) arbitrarily small
as long as \( s\leq \zeta n \) and \( \zeta \) is sufficiently small,
so that \( N_{n,s+1}\leq e^{\hat\zeta n} \). This gives
\[
\sum_{0\leq s\leq\zeta n}|\Q_s^{(n)}| \leq
\frac{C_{1}}{\delta^{*}} \left(\frac{C_{1}\mathcal
D_{\delta}}{\delta}\right)^{\zeta n} e^{\hat \zeta n}
e^{-\gamma_{1} n} |\Delta^{*}|.
\]
Since both \( \zeta \) and \( \hat\zeta \) can be taken arbitrarily
small, the result follows.

\end{proof}

\begin{remark}
Notice that taking \( \zeta \) small (and
thus obtaining \( \hat\zeta \) small) does not depend on choosing \(
\delta \) or \( \delta^{*} \) sufficiently small as \( \zeta \) can be
chosen arbitrarily in the decomposition of the sum in \eqref{set_split}.
\end{remark}

\subsection{Returns after many escapes}

\begin{proof}[Proof of Lemma \ref{manyescapes}]

 To bound the set of points which have many escapes we use a softer
 argument. Clearly we have
 \[
 |\{E_{1}\leq n\}|\leq |\Delta^{*}|.
 \]
 By Lemma \ref{retpart} there is a fixed proportion \( \xi \)
 of \emph{the image} \( f^{E_{1}(\omega)} \) of each escaping
 component which actually has a ``full return'' to \( \Delta^{*} \)
 within \( t^{*} \) iterates after this escape.  Using the bounded
 distortion estimate in Proposition \ref{escapedist} again this
 translates to a minimum proportion \( \xi/\mathcal D_{\delta} \) of
 the actual interval \( \omega \). Therefore the maximum proportion of
 \(  |\{E_{1}\leq n\}| \) which is even allowed potentially to have a
 second escape (whether it be before or after time \( n \)) is bounded
 by \( 1-\xi/\mathcal D_{\delta} < 1 \) and so we have
 \[
 |\{E_{2} \leq n\}|\leq
 |\{E_{2} \text{ defined}: E_{1}\leq n\}| \leq
 \left(1-\frac{\xi}{\mathcal D_{\delta}}\right) |\{E_{1}\leq n\}
 \leq  \left(1-\frac{\xi}{\mathcal D_{\delta}}\right)
 |\Delta^{*}|.
 \]
 Iterating this formula we get
 \[
 |\{E_{s} \leq n\}|
 \leq  \left(1-\frac{\xi}{\mathcal D_{\delta}}\right)^{s-1}
 |\Delta^{*}|
 \]
and so there exists some constant \( C_{4}>0 \) such that
\begin{equation}\label{many_esc}
\sum_{\zeta n<s<n}|\Q_s^{(n)}|\leq \sum_{\zeta
n<s<n}\left(1-\frac{\xi}{\mathcal D_{\delta}}\right)^{s-1}
|\Delta^{*}|
\leq C_{4}\left(1-\frac{\xi}{\mathcal D_{\delta}}\right)^{\zeta n}|\Delta^{*}|.
\end{equation}
The statement in the Lemma follows with \( \gamma_{4}= -\zeta \log
(1-\xi/\mathcal D_{\delta}) \).

\end{proof}

\section{Statistical properties}
\label{sec_tower}

By Proposition \ref{main_tail}, we know the tail of the return times
is decaying exponentially fast and thus \( \mathcal Q \) is indeed a
partition of \( \Delta^{*} \), while by Proposition
\ref{escapedist}, we have the required bounded distortion property.
Thus together, these two Propositions imply Theorem
\ref{thm_markov}. Moreover, the construction and estimates given
above also yield the following additional properties which will be
used below: there exists some \( \lambda'
>1 \) such that
\begin{equation}\label{property1}
d(f^{T}x,f^{T}y)\ge \lambda' d(x,y)
\end{equation}
and there exists a constant \( C>0 \) such that
\begin{equation}\label{property2}
d(f^kx,f^ky)\le Cd(f^{T}x,f^{T}y), \text{ for all } k<T.
\end{equation}
Indeed, \eqref{property1} follows by the same expansion estimates used
repeatedly in the proof and \eqref{property2} follows almost trivially
from the bounded distortion property and
the fact that \( \omega \) never gets bigger than \( \delta \). It
remains to show how these properties imply the required statistical
properties.

\subsection{Absolutely continuous invariant probability measures}
The bounded distortion property of the induced map \( f^{T}:
\Delta^{*}\to\Delta^{*} \) implies, by classical results, the
existence of an ergodic, in fact mixing,
absolutely continuous invariant measure \( \hat\mu \)
for \( f^{T} \) with bounded density with respect to Lebesgue.
Therefore,
the exponential tail of the return time function \( T \)
implies in that
\[
\int T d\mu = \sum_{\omega\in\mathcal Q} T(\omega) \mu(\omega) <\infty.
\]
This means that the measure obtained by \emph{pushing forward} \( \hat\mu \)
by iterates of the original map \( f \) is finite and thus, after
normalization, yields an absolutely continuous probability measure \(
\mu \)
which, again by standard arguments, is also \( f \)-invariant and
ergodic, and in fact mixing for some power of \( f \). The
expansivity estimates obtained above clearly imply that \( \mu \) has
a positive Lypaunov exponent.

\subsection{The Markov extension}
First of all we
define a \emph{Markov extension} or \emph{Markov tower}.
Let
$\mathcal{T}_0$ be a copy $\Delta^{*}$
and let $\mathcal{T}_{j,0}$ be a copy of the element $\omega_j$ of the
partition \( \mathcal Q \). The return time function
$T:\Delta^{*}\to\natural$ is constant on partition elements
$\mathcal{T}_{j,0}$ with value $T|_{\mathcal{T}_{j,0}}=T_j\ge1$. Then
define
\[
\mathcal{T}=\{(x,k):x\in \mathcal{T}_0,\,k=0,\dots,T(x)-1\}=
\bigcup_{j\ge1}\bigcup_{k=0}^{T_j-1}\mathcal{T}_{j,k},
\]
where
\[
\mathcal{T}_{j,k}=\mathcal{T}_{j,0}\times\{k\}.
\]
So
$\mathcal{T}$ is the disjoint union of $T_j$ copies of each
$\mathcal{T}_{j,0}$. Define the {\em tower map}
\[
F:\mathcal{T}\to
\mathcal{T}
\]
by setting
\[
F(x,k)=(x,k+1)
\]
for $0\le k<T(x)-1$ and
\[
F(x,T(x)-1)=(f^Tx,0).
\]
$F$ is
Markov with respect to the partition $\{\mathcal{T}_{j,k}\}$ and the
return map \( F^{T}: \mathcal T_{j, 0} \to \mathcal T_{j,0} \), or
equivalently
$f^T:\Delta^{*}\to\Delta^{*} $,   is
full branched Markov with respect to the partition
$\{\mathcal{T}_{j,0}\}=\mathcal Q$ on \( \mathcal T_{0}=\Delta^{*} \).

\subsection{Projections}
Define the projection
\[
\pi:\mathcal{T}\to \hat{J}
\]
by
$\pi(x,k)=f^k(x)$. Clearly,  $\pi$ is a semi-conjugacy between $F$
and $f$:
\begin{equation}\label{conju}
f \circ \pi = \pi \circ F.
\end{equation}
Using the projection \( \pi \), any measure \( \tilde\mu \)
on \( \mathcal T \) can be
projected to a measure \(  \mu = \pi_{*}\tilde\mu\)
on \( \hat J \) with the property that \( \mu(A) =
\tilde\mu(\pi^{-1}(A)) \) for every measurable
set \( A \subset \hat J \). Moreover any
observable
\[
\varphi: \hat
J \to \mathbb R
\]
can be lifted to an observable
\[
\tilde\varphi  : \mathcal T \to \mathbb R
\quad \text{ given by }\quad
\tilde\varphi = \varphi \circ \pi.
\]
Notice that \eqref{conju} implies
\[
\tilde \varphi \circ F^{n} = \varphi \circ f^{n} \circ \pi \quad
\text{ and }\quad
 \int \tilde\varphi d\tilde\mu = \int \varphi \circ \pi d\tilde\mu
= \int \varphi d(\pi_{*}\tilde\mu).
\]
Therefore, for any two observables \( \varphi, \psi : \hat J \to
\mathbb R \) and their corresponding lifts \(  \tilde\varphi, \tilde
\psi : \mathcal T \to \mathbb R \), any measure \( \tilde \mu \) and
its corresponding projection \( \mu = \pi_{*}\tilde \mu \) it is
sufficient to prove statistical properties for the lifts on the
tower to obtain similar results for the original map using the fact
that \( \pi \) is a ``measure-preserving'' (by definition, since one
measure is the projection of the other by \( \pi \)) semi-conjugacy.
For example, for the correlation function we have
\begin{align*}
    \int (\tilde\varphi \circ F^{n}) \tilde\psi d\tilde\mu -
    \int \tilde\varphi d\tilde\mu\int \tilde\psi d\tilde\mu
& = \int(\varphi\circ (f^n\circ\pi))(\psi\circ\pi)d\tilde\mu
-\int(\varphi\circ \pi )d\tilde\mu\int(\psi\circ\pi)d\tilde\mu
\\& =\int(\varphi\circ f^n)\psi d(\pi_{*}\tilde\mu) -\int\varphi
d(\pi_{*}\tilde\mu) \int\psi d(\pi_{*}\tilde\mu)
 \\ &   =
 \int (\varphi \circ f^{n}) \psi d\mu - \int \varphi d\mu\int \psi d\mu.
\end{align*}
An important observation however is that the properties must be proved
on the tower for all observables which can be obtained as lifts of
suitable (H\"older continuous) observables on the manifold.

\subsection{Statistical properties on the tower}
In this section we state precisely the results of Young
\cite{You99} which give conditions for the exponential decay of
certain classes of observables for the map \( F:\mathcal T \to
\mathcal T \) and for the Central Limit Theorem.
Young's setting is very general and only requires \(
\Delta \) to be a measure space with some reference measure \( m \)
and the Markov tower structure described above. The Markov
structure naturally induces a symbolic metric on \( \mathcal T \)
which can be used to define notions of regularity both for the
Jacobian of \( F \) and for observables on \( \mathcal T \).

\subsubsection{Symbolic metric}
First of all we define a \emph{separation time}:
if $x,y\in\mathcal{T}_{j,0}$ for some $j$, then
\[
s(x,y)= \max\left\{ n\ge0 \text{ s.t } (F^T)^nx \text{ and }
(F^T)^ny \text{ lie in the same partition element of }
\mathcal{T}_0\right\}\]
If $x,y$ lie in distinct
partition elements, then $s(x,y)=0$. If $x,y\in\mathcal{T}_{j,k}$
then write $x=F^k x_0$, $y=F^k y_0$ where
$x_0,y_0\in\mathcal{T}_{j,0}$ and define $s(x,y)=s(x_0,y_0)$.
Then, for a fixed constant \( \sigma\in (0,1) \),
we define
\[
d_\sigma(x,y)=\sigma^{s(x,y)}.
\]
By the expansivity property of \( F^{T} \) any two distinct points
eventually separate and therefore this defines a metric on \( \mathcal T \).

\subsubsection{Regularity of observables in the symbolic metric}
Using the metric \( d_{\sigma} \) we  define
\[
\mathcal{H}_{\sigma}(\mathbb{R})=
\{ \phi: \T \rightarrow \mathbb{R}: \exists C_{\phi}:
|\phi(x)-\phi(y)| \leq C_{\phi} \sigma^{s(x,y)} \ \forall \
x,y\in\mathcal T \}
\]
to be  the  space of
``H\"older continuous'' functions on $\T$.

\subsubsection{Young's Theorem}
Let \( F: \mathcal T \to \mathcal T \) be a Markov map on a tower \(
\mathcal T \) with a reference measure \( m \), a Markov return map
\( F^{T}:\mathcal T_{0}\to\mathcal T_{0} \), an exponential tail of
the return time function \( T: \mathcal T_{0}\to \mathbb N \), and
such that the greatest common divisor of all values taken by the
function \( T \) is 1. Let \( JF^{T}\) denote the Jacobian of \(
F^{T} \) with respect to the reference measure \( m \).
\begin{theorem*}[\cite{You99}]
 Suppose that the following bounded distortion condition holds:
there exists a constant \( \tilde C \) such that for all \(
x,y\in\mathcal T_{0} \) we have
\begin{equation}\label{distsymb}
\left|\frac{|JF^{T}(x)|}{|JF^{T}(y)|}-1\right|\leq \tilde C
\sigma^{s(x,y)}.
\end{equation}
Then the correlation function \( C_{n}(\tilde\varphi
,  \tilde\psi,\tilde{\mu} )\)
decays exponentially fast (at a uniform rate)
for every observable \( \tilde\varphi\in
L^{\infty}(\tilde\mu) \) and \( \tilde\psi \in \mathcal{H}_{\sigma} \).
Moreover, the Central Limit Theorem holds for every
\( \tilde\psi \in \mathcal{H}_{\sigma} \).
 \end{theorem*}

Notice that the bounded distortion condition \eqref{distsymb} is the
\emph{only} assumption of the theorem over and above the
Markov tower structure, the exponential decay of the return time
function, and the assumption on the greatest common divisor of \( T
\). Moreover, both the exponential tail and the assumption on the gcd
can be relaxed to some extent. If the gcd of \(
T \) is \( k>1 \) we can just consider \( \tilde F=F^{k}:\mathcal T
\to \mathcal T \) which will continue to have the same Markov
structure and properties and a return time function \( \tilde T \)
with gcd (\( \tilde T = 1\)). The exponential decay of the tail of the
return time function can also be relaxed to assume only subexponential
or even polynomial decay (in fact including these cases
is one of the main motivations
for \cite{You99}), although in this case the rate of decay of
the correlation function is correspondingly slower.

\subsection{Decay of correlations for the original map}
To obtain our desired results for \( f \) we therefore just need to
show that the   the bounded distortion condition
of Proposition \ref{escapedist} implies the required bounded
distortion condition \eqref{distsymb} in the symbolic metric, and that
observables on \( \hat J \) which are H\"older
continuous with respect to the usual
Euclidean metric lift to bounded observables on \( \mathcal T \)
which are  H\"older continuous in the symbolic metric.

\begin{lemma}
    The bounded distortion condition \eqref{distsymb} holds.
 \end{lemma}
 \begin{proof}
For the bounded distortion condition, notice first of all that \( F^{T} \)
is differentiable in our setting and therefore the left hand side of
\eqref{distsymb} formulated in terms of the Jacobian of \( F \) is
exactly the same as the left hand side of the second inequality in
Proposition \ref{escapedist} formulated in terms of the derivative of \(
f \), with \( k=T \). Thus all we need to show is that
\[
\tilde{\mathcal D}{} |f^{T}(x)-f^{T}(y)| \leq
C\sigma^{s(x,y)}|\Delta^{*}|
\]
for any two points \( x,y \) belonging to the same element of the
partition \( \mathcal Q = \mathcal T_{0} \). To see this, let \(
\lambda'  \) be the expansion constant in \eqref{property1}, then
cylinder sets are shrinking exponentially at rate \( \lambda' \) and
so we have
\begin{equation}\label{cylinder}
|x-y| \leq |\Delta^{*}|\lambda'^{-s(x,y)} \quad\text{ and so } \quad
|f^{T}(x)-f^{T}(y)| \leq |\Delta^{*}|\lambda'^{-s(x,y)+1}.
\end{equation}
Thus the result follows as long as \( 1> \sigma >  \lambda'^{-1}. \)
\end{proof}
\begin{lemma}
    If $\varphi$ is H\"older continuous with exponent
    $\gamma\geq \tilde\gamma$ then
    $\tilde\varphi=\pi\circ\varphi\in\mathcal{H}_{\sigma}.$
 \end{lemma}
 \begin{proof}
Let \( \varphi: \hat J \to
\mathbb R \) be H\"older continuous: there exists \( C_{\varphi} \)
such that
\[
|\varphi(x)-\varphi(y)|\leq C_{\varphi} |x-y|^{\alpha}
\]
for any two \( x,y\in\hat J \). Now, for any \( \tilde x, \tilde y \)
belonging to some partition element of
\( \mathcal T\) we have
\[
|\tilde\varphi(\tilde x) - \tilde\varphi (\tilde y)| = |\varphi
(\pi(\tilde x)) - \varphi (\pi (\tilde y))| \leq
C_{\varphi}|\pi(x)-\pi(y)|^{\alpha}.
\]
Since \( \pi(x) \) and \( \pi(y) \) belong to the image \( \omega_{k} \)
of some element \( \omega\in\mathcal Q \) for \( k<T \),
then \eqref{property2} and \eqref{cylinder}
give
\[
|\pi(x)-\pi(y)|^{\alpha} \leq
C^{\alpha}|f^{T-k}(\pi(x))-f^{T-k}(\pi(y))|^{\alpha} \leq
|C^{\alpha}\Delta^{*}|^{\alpha}\lambda'^{\alpha(-s(x,y)+1)}.
\]
The result follows if \( 1>\sigma > \lambda'^{-\alpha} \).
\end{proof}

\end{document}